\documentclass[12pt,reqno]{amsart}
\usepackage{amsthm, amscd, amsfonts, amssymb, graphicx, color}
\usepackage[bookmarksnumbered, colorlinks, plainpages,linkcolor=green,anchorcolor=blue,citecolor=blue,urlcolor=blue]{hyperref}
\usepackage{enumerate}
\usepackage{exscale}
\usepackage[all]{xy}
\usepackage{relsize}
\usepackage{pdfsync}
\numberwithin{equation}{section}

\makeatletter
\@namedef{subjclassname@2020}{%
  \textup{2020} Mathematics Subject Classification}
\makeatother

\usepackage[includemp,body={398pt,550pt},footskip=30pt,%
            marginparwidth=60pt,marginparsep=10pt]{geometry}

\newtheorem{theorem}{Theorem}[section]
\newtheorem{lemma}[theorem]{Lemma}
\newtheorem{corollary}[theorem]{Corollary}

\newcommand{\esssup}{\operatorname{ess\,sup}}

\newcommand{\D}{\mathbb{D}}
\newcommand{\DD}{\widehat{\mathcal{D}}}

\newcommand{\DDD}{\mathcal{D}}
\newcommand{\N}{\mathbb{N}}

\newcommand{\C}{\mathbb{C}}

\renewcommand{\phi}{\varphi}

\newcommand{\T}{\mathbb{T}}

\begin{document}
\title[Embedding theorems and integration operators]{Embedding theorems and integration operators on Hardy--Carleson type tent spaces induced by doubling weights}
\date{}

\author[J. Chen]{Jiale Chen}
\thanks{The first author is supported by National Natural Science Foundation of China (No. 12501170). The second author is supported by National Natural Science Foundation of China (No. 12401091).}
\address{Jiale Chen, School of Mathematics and Statistics, Shaanxi Normal University, Xi'an 710119, China.}
\email{jialechen@snnu.edu.cn}

\author[B. Liu]{Bin Liu\textsuperscript{*}}
\thanks{\textsuperscript{*}Corresponding author.}
\address{Bin Liu, School of Mathematical Sciences, Chengdu University of Technology, Sichuan, 610059, China.}
\email{liubin\_taixi@163.com}


\begin{abstract}
  \noindent This paper develops the function and operator theory of Hardy--Carleson--type analytic tent spaces $AT_q^\infty(\omega)$ induced by radial weights $\omega$ satisfying a two-sided doubling condition. We first characterize the positive Borel measures $\mu$ for which the embedding from $AT_p^\infty(\omega)$ into the tent space $T_q^\infty(\mu)$ is bounded for all $0 < p, q < \infty$. A Littlewood--Paley formula for $AT_q^\infty(\omega)$ is then established. Using these results, we give a complete characterization of the boundedness (compactness) of Volterra-type integration operators between $AT_p^\infty(\omega)$ and $AT_q^\infty(\omega)$.
\end{abstract}
\maketitle

\section{Introduction}
\allowdisplaybreaks[4]

Let $\D$ be the open unit disk of the complex plane $\C$ and $\T:=\partial\D$ be the unit circle. Given $0<p,q<\infty$ and a positive Borel measure $\mu$ on $\D$, the tent space $T_q^p(\mu)$ consists of $\mu$-measurable functions $f$ on $\D$ with
$$\|f\|_{T^p_q(\mu)}:=\left(\int_{\T}\left(\int_{\Gamma(\xi)}|f(z)|^q\frac{d\mu(z)}{1-|z|}\right)^{\frac{p}{q}}|d\xi|\right)^{\frac{1}{p}}<\infty,$$
where $\Gamma(\xi)$ is the Kor\'anyi (non-tangential approach) region with vertex at $\xi\in\T$ defined by
$$\Gamma(\xi): = \{ z \in \D : |1 - \bar{\xi}z| < 1 - |z|^2 \}.$$
In the case $p=\infty$, the tent space $T^{\infty}_q(\mu)$ consists of $\mu$-measurable functions $f$ satisfying
$$\|f\|_{T^{\infty}_q(\mu)}:=\mathrm{ess}\sup_{\xi\in\T}\left(\sup_{a\in\Gamma(\xi)}\frac{1}{1-|a|}\int_{S(a)}|f(z)|^qd\mu(z)\right)^{\frac{1}{q}}<\infty,$$
where $S(a)$ is the Carleson square generated by $a$, defined by $S(0):=\D$ and
$$S(re^{i\theta}):=\left\{\rho e^{it}:r\leq \rho<1, |t-\theta|\leq\frac{1-r}{2}\right\}$$
for $a=re^{i\theta}\neq0$. We also use $T^0_q(\mu)$ to denote the subspace of $T^{\infty}_q(\mu)$ consisting of $f$ with
$$\lim_{|a|\to1^-}\frac{1}{1-|a|}\int_{S(a)}|f(z)|^qd\mu(z)=0.$$

A function $\omega : \mathbb{D} \to [0, \infty)$ that is integrable over $\mathbb{D}$ is called a weight. It is said to be radial if $\omega(z) = \omega(|z|)$ for all $z \in \mathbb{D}$. When a measure $\mu$ is induced by a radial weight $\omega$ via  
\begin{equation*} 
d\mu = \omega dA,
\end{equation*}  
we write $T_q^p(\mu) := T_q^p(\omega),\; T_q^\infty(\mu) := T_q^\infty(\omega)$, and $T_q^0(\mu) := T_q^0(\omega)$. Here and throughout, $dA$ denotes the Lebesgue area measure on $\mathbb{D}$. In particular, for the standard weight $\omega(z) = (1 - |z|)^\alpha$ with some $\alpha \in \mathbb{R}$, the corresponding tent spaces are denoted by $T_{q,\alpha}^p,\; T_{q,\alpha}^\infty$, and $T_{q,\alpha}^0$. The main focus of this paper is the function and operator theory on analytic tent spaces. Let $\mathcal{H}(\mathbb{D})$ denote the space of analytic functions on $\mathbb{D}$. The Hardy-type analytic tent space $AT_q^p(\omega)$ is defined as the intersection of $T_q^p(\omega)$ with $\mathcal{H}(\mathbb{D})$, endowed with the inherited (quasi)-norm. Similarly, the Hardy--Carleson-type analytic tent spaces $AT_q^\infty(\omega)$ and $AT_q^0(\omega)$ are defined as the intersections of $T_q^\infty(\omega)$ and $T_q^0(\omega)$, respectively, with $\mathcal{H}(\mathbb{D})$, again equipped with the inherited (quasi)-norm.

For a radial weight $\omega $, write $\widehat{\omega} (z)=\int_{\left | z \right | }^{1} \omega (s)ds$ for all $z \in \mathbb{D} $. A radial weight $\omega $ belongs to the class $\widehat{\mathcal{D} } $ if there exists a constant $C=C(\omega)\ge 1$ such that
$$\widehat{\omega} (r)\le C\widehat{\omega} \left(\frac{1+r}{2}\right)$$
for all $0\le r< 1  $. Moreover, a radial weight $\omega $ belongs to the class $\check{\mathcal{D} } $ if there exist  $K=K(\omega)> 1  $ and $C=C(\omega)> 1$ such that 
\begin{equation}\label{check}
\widehat{\omega} (r) \geq C\widehat{\omega} \left(1-\frac{1-r}{K} \right)
\end{equation}
for all $0\le r< 1  $. The intersection $\check{\mathcal{D} } \cap \widehat{\mathcal{D} } $ is denoted by $\mathcal{D}  $, and this is the class of weights that we mainly work with. For basic properties of weights from the above classes, see \cite{Pe16,PR14,PR21} and the references therein.

The theory of tent spaces, introduced by Coifman, Meyer and Stein in their seminal work \cite{CMS85}, has become a fundamental tool in harmonic analysis, complex analysis and operator theory. These spaces naturally extend several classical function spaces. For instance, with appropriate choices of $\omega$ and $p,q$, one recovers Hardy spaces, Bergman spaces and the space of bounded mean oscillation. Moreover, the spaces $T^{\infty}_q(\mu)$ naturally relate to the classical Carleson measures. The study of such spaces has attracted considerable attention in recent years; see, for example, \cite{CN,CPW,LP24-1,LPW24,Mu20,PP20} for applications of tent space theory in operator theory on Hardy and Bergman spaces, \cite{LP24,PR15,PRS15} for the systematic treatment of weighted Bergman spaces via tent space methods, and \cite{AMPR,Ch23,CW22,Per22,WZ21,WZ22} for the development of function and operator theory on analytic tent spaces. In particular, Aguilar-Hernández, Mas, Pel\'aez and R\"atty\"a \cite{AMPR} recently explored some properties of the analytic tent spaces $T_q^p(\omega)$, including the Littlewood--Paley formulas and boundedness of integration operators on these spaces.

In this paper, we are going to investigate the fundamental function and operator theory on analytic Hardy–Carleson--type tent spaces $AT_p^\infty(\omega)$, where the weight $\omega$ belongs to the class $\mathcal{D}$. We first establish sharp embedding theorems for $AT_p^\infty(\omega)$—including characterizations of bounded embeddings into weighted tent spaces $T_q^\infty(\mu)$—in terms of Carleson-type conditions on the measure $\mu$ relative to $\omega$. Next, we obtain a derivative characterization of $AT_q^\infty(\omega)$, which yields a complete description of the boundedness and compactness of Volterra-type integration operators acting between such spaces.

To state our results, recall that the pseudo-hyperbolic metric $d(\cdot,\cdot)$ on $\D$ is defined by
$$d(z,u):=\left|\frac{z-u}{1-z\overline{u}}\right|,\quad z,u\in\D.$$
For $z\in\D$ and $r\in(0,1)$, we use $\Delta(z,r):=\{u\in\D:d(z,u)<r\}$ to denote the pseudo-hyperbolic metric disk centered at $z$ with radius $r$. We are now ready to state our first result, which completely characterizes the boundedness of the embeddings $I_d:AT^{\infty}_p(\omega)\to T^{\infty}_q(\mu)$ and $I_d:AT^0_p(\omega)\to T^0_q(\mu)$.

\begin{theorem}\label{bounded-embedding}
Let $0<p, q<\infty$, $\omega\in\DDD$, and let $\mu$ be a positive Borel measure on $\D$. For $r\in(0,1)$, write
$$G_{\mu,r}(z):=\frac{\mu(\Delta(z,r))^{\frac{1}{q}}}{\widehat{\omega}(z)^{\frac{1}{p}}(1-|z|)^{\frac{1}{q}}}.$$
\begin{enumerate}
  \item [(1)] If $p\leq q$, then the following conditions are equivalent:
      \begin{enumerate}
          \item [(a)] $I_d:AT^{\infty}_p(\omega)\to T^{\infty}_q(\mu)$ is bounded;
          \item [(b)] $I_d:AT^0_p(\omega)\to T^0_q(\mu)$ is bounded;
          \item [(c)] for some (or any) $r\in(0,1/4)$, $G_{\mu,r}\in L^{\infty}$.
      \end{enumerate}
      Moreover, in these cases,
      $$\|I_d\|\asymp\left\|G_{\mu,r}\right\|_{L^{\infty}}.$$
  \item [(2)] If $p>q$, then the following conditions are equivalent:
      \begin{enumerate}
          \item [(a)] $I_d:AT^{\infty}_p(\omega)\to T^{\infty}_q(\mu)$ is bounded;
          \item [(b)] $I_d:AT^0_p(\omega)\to T^0_q(\mu)$ is bounded;
          \item [(c)] for some (or any) $r\in(0,1/4)$, $G_{\mu,r}\in T^{\infty}_{\frac{pq}{p-q},-1}$.
      \end{enumerate}
      Moreover, in these cases,
      $$\|I_d\|\asymp\left\|G_{\mu,r}\right\|_{T^{\infty}_{\frac{pq}{p-q},-1}}.$$
\end{enumerate}
\end{theorem}

The above theorem constitutes a Carleson-type embedding theorem for the analytic tent spaces $AT^{\infty}_p(\omega)$. Its proof relies on fundamental structural properties of these spaces—including growth estimates—as well as the Carleson measure characterization of the (non-analytic) tent spaces $T^{\infty}_q(\mu)$. Furthermore, we establish a “little-oh” counterpart to Theorem \ref{bounded-embedding}: specifically, we provide a complete characterization of the compactness of the identity operator $I_d$ acting from $AT^{\infty}_p(\omega)$ to $T^{\infty}_q(\mu)$ and from $AT^0_p(\omega)$ to $T^0_q(\mu)$. This characterization is based on the atomic decompositions of $AT^{\infty}_p(\omega)$ and $AT^0_p(\omega)$.

Our next result is the following Littlewood--Paley formula, which provides equivalent (quasi-)norms for the spaces $AT^{\infty}_q(\omega)$ in terms of the iterated derivatives.

\begin{theorem}\label{Littlewood--Paley}
Let $0<q<\infty$, $m\in\mathbb{N}$, and let $\omega$ be a radial weight. Then the following conditions are equivalent:
\begin{enumerate}
    \item [(a)] for any analytic function $f$ on $\D$, $f\in AT^{\infty}_q(\omega)$ if and only if the function $f^{(m)}(\cdot)(1-|\cdot|)^m$ belongs to $T^{\infty}_q(\omega)$, and
    \begin{equation}\label{LP}
    \|f\|_{AT^{\infty}_q(\omega)}\asymp\sum_{j=0}^{m-1}\left|f^{(j)}(0)\right|+\left\|f^{(m)}(\cdot)(1-|\cdot|)^m\right\|_{T^{\infty}_q(\omega)};
    \end{equation}
   \item [(b)] $\omega\in\DDD$.
\end{enumerate}
\end{theorem}

The implication (a)$\Rightarrow$(b) of the above theorem follows by testing \eqref{LP} with monomials. When establishing the implication (b)$\Rightarrow$(a), a generalized Forelli--Rudin-type estimate involving weights from $\DDD$ will play an essential role.

With Theorems \ref{bounded-embedding} and \ref{Littlewood--Paley} in hand, we can determine the bounded Volterra--type operators acting between $AT^{\infty}_p(\omega)$ and $AT^{\infty}_q(\omega)$. Given an analytic symbol $g$ on $\D$, the Volterra--type integration operator $J_g$ is defined by
$$
J_g f(z) = \int_0^z f(\zeta) g'(\zeta) d\zeta, \quad z \in \mathbb{D},\quad f\in\mathcal{H}(\D).
$$
The boundedness, compactness and spectral properties of $J_g$ on various spaces---such as Hardy spaces, Bergman spaces, and $BMOA$---have been extensively studied; we refer to \cite{ACS01,AS97,Mu20,Pa16} and the references therein. In the setting of tent spaces, Wang and Zhou \cite{WZ22} characterized the boundedness and compactness of $J_g$ acting between analytic Hardy type tent spaces induced by standard weights. The first author and Wang \cite{CW22} investigated the boundedness, compactness and strict singularity of $J_g$ acting on analytic Hardy--Carleson--type tent spaces induced by standard weights. Very recently, Aguilar-Hern\'andez, Mas, Pel\'{a}ez and R\"atty\"a \cite{AMPR} decribed the boundedness of $J_g$ acting on the analytic tent spaces $AT^p_q(\omega)$ with $\omega\in\DDD$. We here characterize the boundedness of $J_g:AT^{\infty}_p(\omega)\to AT^{\infty}_q(\omega)$ for all $0<p,q<\infty$ and $\omega\in\DDD$.

\begin{theorem}\label{bounded-Volterra}
Let $0<p,q<\infty$, $\omega\in\DDD$, and let $g$ be an analytic function on $\D$.
\begin{enumerate}
    \item [(1)] If $p\leq q$, then the following conditions are equivalent:
      \begin{enumerate}
        \item [(a)] $J_g:AT^{\infty}_p(\omega)\to AT^{\infty}_q(\omega)$ is bounded;
        \item [(b)] $J_g:AT^0_p(\omega)\to AT^0_q(\omega)$ is bounded;
        \item [(c)] $\sup_{z\in\D}\left|g'(z)\right|(1-|z|)\widehat{\omega}(z)^{\frac{1}{q}-\frac{1}{p}}<\infty$.
      \end{enumerate}
      Moreover,
      $$\|J_g\|\asymp
      \sup_{z\in\D}\left|g'(z)\right|(1-|z|)\widehat{\omega}(z)^{\frac{1}{q}-\frac{1}{p}}.$$
    \item [(2)] If $p>q$, then the following conditions are equivalent:
      \begin{enumerate}
          \item [(a)] $J_g:AT^{\infty}_p(\omega)\to AT^{\infty}_q(\omega)$ is bounded;
          \item [(b)] $J_g:AT^0_p(\omega)\to AT^0_q(\omega)$ is bounded;
          \item [(c)] $g\in AT^{\infty}_{\frac{pq}{p-q}}(\omega)$.
      \end{enumerate}
      Moreover,
      $$\|J_g\|\asymp\|g-g(0)\|_{AT^{\infty}_{\frac{pq}{p-q}}(\omega)}.$$
\end{enumerate}
\end{theorem}

This paper is organized as follows. In Section \ref{preliminary}, we give some preliminary results involving Carleson measures and radial weights. In particular, we establish some generalized Forelli--Rudin-type estimates for the radial weights from $\DD$ and $\DDD$ that play an important role in this paper. In Section \ref{atomic-decomposition}, we establish the atomic decomposition for the analytic tent spaces $AT^{\infty}_q(\omega)$ and $AT^0_q(\omega)$. Section \ref{embedding-theorems} is devoted to the proof of Theorem \ref{bounded-embedding}. Furthermore, the compactness of the embeddings $I_d:AT^{\infty}_p(\omega)\to T^{\infty}_q(\mu)$ and $I_d:AT^0_p(\omega)\to T^0_q(\mu)$ is also characterized. Finally, we prove Theorems \ref{Littlewood--Paley} and \ref{bounded-Volterra} in Section \ref{LP-Volterra}.

Throughout the paper, we write $A\lesssim B$ (or $B\gtrsim A$) to denote that there exists a nonessential constant $C>0$ such that $A\leq CB$. If $A\lesssim B\lesssim A$, then we write $A\asymp B$. For a subset $E\subset \D$, $\chi_E$ denotes its characteristic function.

\section{Preliminaries}\label{preliminary}

In this section, we introduce some preliminary results that will be used throughout the paper.

\subsection{Carleson measures}

The concept of Carleson measures was introduced by Carleson \cite{Ca58,Ca62} when studying interpolating sequences for bounded analytic functions on the unit disk, which had become an important tool in the area of operator theory on analytic function spaces. A positive Borel measure $\nu$ on $\D$ is said to be a Carleson measure if
$$\|\nu\|_{CM}:=\sup_{a\in\D}\frac{\nu(S(a))}{1-|a|}<\infty.$$
It is well-known (see for instance \cite[Theorem 45]{ZZ08}) that $\nu$ is a Carleson measure if and only if for some (or any) $s>0$ one has
$$\sup_{a\in\D}\int_{\D}\frac{(1-|a|)^s}{|1-\overline{a}z|^{s+1}}d\nu(z)<\infty.$$
Moreover, with constant depending on $s$, the supremum of the above integral is comparable to $\|\nu\|_{CM}$.

A positive Borel measure $\nu$ on $\D$ is said to be a vanishing Carleson measure if
$$\lim_{|a|\to1^-}\frac{\nu(S(a))}{1-|a|}=0.$$
Similarly as before, $\nu$ is a vanishing Carleson measure if and only if for some (or any) $s>0$,
$$\lim_{|a|\to1^-}\int_{\D}\frac{(1-|a|)^s}{|1-\overline{a}z|^{s+1}}d\nu(z)=0.$$

Let $\mu$ be a positive Borel measure on $\D$ and $0<q<\infty$. By the definition of the tent spaces $T^{\infty}_q(\mu)$ and $T^0_q(\mu)$, we see that a measurable function $f$ belongs to $T^{\infty}_q(\mu)$ if and only if the measure $d\nu_{\mu,f}:=|f|^qd\mu$ is a Carleson measure, and it belongs to $T^0_q(\mu)$ if and only if $\nu_{\mu,f}$ is a vanishing Carleson measure. Moreover,
$$\|f\|_{T^{\infty}_q(\mu)}=\left\|\nu_{\mu,f}\right\|_{CM}^{1/q}.$$
Throughout the paper, we will use these characterizations repeatedly.

\subsection{Estimates on weights}

In this subsection, we collect some integral estimates involving doubling weights. We begin with the following classical Forelli--Rudin estimate, which can be found in \cite[Lemma 3.10]{Zh07}.

\begin{lemma}[\cite{Zh07}]\label{FR}
Let $\alpha>-1$ and $s>0$. Then for $z\in\D$,
$$\int_{\D}\frac{(1-|u|)^{\alpha}}{|1-\overline{u}z|^{2+\alpha+s}}dA(u)\asymp(1-|z|)^{-s}.$$
\end{lemma}

The following lemma gives some characterizations of weights in $\DD$, which was proved in \cite[Lemma 2.1]{Pe16}.

\begin{lemma}[\cite{Pe16}]\label{elementary}
Let $\omega$ be a radial weight. Then the following conditions are equivalent:
\begin{enumerate}
    \item [(a)] $\omega\in \DD$;
    \item [(b)] there exist $C=C(\omega)>0$ and $\beta_0=\beta_0(\omega)>0$ such that for all $\beta\geq\beta_0$,
    $$\widehat{\omega}(r)\leq C\left(\frac{1-r}{1-t}\right)^{\beta}\widehat{\omega}(t),\quad 0\leq r\leq t<1;$$
    \item [(c)] there exist $C=C(\omega)>0$ and $\gamma_0=\gamma_0(\omega)>0$ such that for all $\gamma\geq\gamma_0$,
    $$\int_0^t\left(\frac{1-t}{1-r}\right)^{\gamma}\omega(r)dr\leq C\widehat{\omega}(t),\quad 0\leq t<1.$$
\end{enumerate}
\end{lemma}

The following description of weights from $\check{\mathcal{D}}$ was established in the proof of \cite[Theorem 15]{PR21}.

\begin{lemma}[\cite{PR21}]\label{elementary-Dcheck}
Let $\omega$ be a radial weight. Then $\omega\in\check{\mathcal{D}}$ if and only if there exist $C=C(\omega)>0$ and $\alpha_0=\alpha_0(\omega)>0$ such that
$$\widehat{\omega}(t)\leq C\left(\frac{1-t}{1-r}\right)^{\alpha_0}\widehat{\omega}(r),\quad 0\leq r\leq t<1.$$
\end{lemma}

Based on Lemmas \ref{elementary}(b) and \ref{elementary-Dcheck}, we obtain that if $\omega\in\DDD$, then for any $r\in(0,1)$, there exists $C=C(\omega,r)\geq1$ such that
\begin{equation}\label{D-constant}
C^{-1}\widehat{\omega}(u)\leq \widehat{\omega}(z)\leq C\widehat{\omega}(u)\quad \mathrm{whenever} \quad d(z,u)<r.
\end{equation}

The following integral estimate can be found in \cite[Lemma 5]{PPR20}.

\begin{lemma}[\cite{PPR20}]\label{hat-equivalent}
Let $\alpha\geq0$, $0<p<\infty$ and $\omega\in\DDD$. Then for any analytic function $f$ on $\D$,
$$\int_{\D}|f(z)|^p(1-|z|)^{\alpha}\omega(z)dA(z)
\asymp\int_{\D}|f(z)|^p(1-|z|)^{\alpha-1}\widehat{\omega}(z)dA(z).$$
\end{lemma}

The following lemma establishes some generalized Forelli--Rudin-type estimates involving doubling weights, which will play an important role in the sequel.

\begin{lemma}\label{gFR}
Let $s,\gamma>1$, $\delta\in\mathbb{R}$, and let $\omega$ is a radial weight. For $a,z\in\D$, consider the integral
$$I(a,z):=\int_{\D}\frac{(1-|u|)^{-\delta}\omega(u)dA(u)}{|1-\overline{a}u|^s|1-\overline{z}u|^{\gamma}}.$$
\begin{enumerate}
    \item [(1)] If $\omega\in\DD$, $\gamma>\max\{\gamma_0(\omega),1\}$ and $\delta=0$, then
    $$I(a,z)\lesssim\frac{\widehat{\omega}(z)}{(1-|a|)^{s-1}(1-|z|)^{\gamma}}.$$
    \item [(2)] If $\omega\in\DDD$ and $\gamma\geq m+1+\gamma_0(\omega)$ for some nonnegative integer $m$, then there exists $\delta_0=\delta_0(\omega)>0$ such that for $s\in(1,1+\delta_0)$ and $\delta\in[-m,\delta_0)$,
    $$I(a,z)\lesssim\frac{\widehat{\omega}(z)}{|1-\overline{a}z|^s(1-|z|)^{\gamma+\delta-1}}.$$
\end{enumerate}
\end{lemma}
\begin{proof}
Since $s,\gamma>1$, we can use \cite[Proposition 3.2]{ZGCT} to obtain that
\begin{align*}
I(a,z)&=\int_{\D}\frac{(1-|u|)^{-\delta}\omega(u)dA(u)}{|1-\overline{a}u|^s|1-\overline{z}u|^{\gamma}}\\
&\leq\int_0^1\int_{\T}\frac{|d\xi|}{|1-r\overline{a}\xi|^s|1-r\overline{z}\xi|^{\gamma}}(1-r)^{-\delta}\omega(r)dr\\
&\asymp\int_0^1\frac{(1-r)^{-\delta}\omega(r)dr}{(1-r^2|a|^2)^{s-1}|1-r^2\overline{a}z|^{\gamma}}+\int_0^1\frac{(1-r)^{-\delta}\omega(r)dr}{(1-r^2|z|^2)^{\gamma-1}|1-r^2\overline{a}z|^s}\\
&=:I_1(a,z)+I_2(a,z).
\end{align*}

(1) If $\omega\in\DD$, $\gamma>\max\{\gamma_0(\omega),1\}$ and $\delta=0$, then by Lemma \ref{elementary}(c),
\begin{align*}
I_1(a,z)&\leq\frac{1}{(1-|a|)^{s-1}}\left(\int_0^{|z|}+\int_{|z|}^1\right)\frac{\omega(r)dr}{(1-r|z|)^{\gamma}}\\
&\leq\frac{1}{(1-|a|)^{s-1}}\left(\int_0^{|z|}\frac{\omega(r)dr}{(1-r)^{\gamma}}+\frac{\widehat{\omega}(z)}{(1-|z|)^{\gamma}}\right)\\
&\lesssim\frac{\widehat{\omega}(z)}{(1-|a|)^{s-1}(1-|z|)^{\gamma}}.
\end{align*}
Similarly,
\begin{align*}
I_2(a,z)&\leq\frac{1}{(1-|a|)^{s-1}}\left(\int_0^{|z|}+\int_{|z|}^1\right)\frac{\omega(r)dr}{(1-r|z|)^{\gamma}}\\
&\lesssim\frac{\widehat{\omega}(z)}{(1-|a|)^{s-1}(1-|z|)^{\gamma}}.
\end{align*}

(2) Let $\omega\in\DDD$ and $\gamma\geq m+1+\gamma_0(\omega)$ for some nonnegative integer $m$. We first claim that there exists $\tau_0=\tau_0(\omega)>0$ such that for any $\tau\in(0,\tau_0)$,
\begin{equation}\label{hat-integral}
\int_{|z|}^1\frac{\widehat{\omega}(r)dr}{(1-r)^{1+\tau}}\lesssim\frac{\widehat{\omega}(z)}{(1-|z|)^{\tau}}.
\end{equation}
In fact, since $\omega\in\DDD\subset\check{\mathcal{D}}$, there exist $C=C(\omega)>1$ and $K=K(\omega)>1$ such that \eqref{check} holds for any $0\leq r<1$.
Choose $\tau_0(\omega)=\frac{\log C(\omega)}{\log K(\omega)}$. Then for any $\tau\in(0,\tau_0)$, $K(\omega)^{\tau}<C(\omega)$, and we have
$$\int_{|z|}^1\frac{\widehat{\omega}(r)dr}{(1-r)^{1+\tau}}\geq C\int_{|z|}^1\frac{\widehat{\omega}\left(1-\frac{1-r}{K}\right)}{(1-r)^{1+\tau}}dr
=\frac{C}{K^{\tau}}\int_{1-\frac{1-|z|}{K}}^1\frac{\widehat{\omega}(r)dr}{(1-r)^{1+\tau}},$$
which implies that
\begin{align*}
\int_{|z|}^1\frac{\widehat{\omega}(r)dr}{(1-r)^{1+\tau}}
&=\left(\int_{|z|}^{1-\frac{1-|z|}{K}}+\int_{1-\frac{1-|z|}{K}}^1\right)\frac{\widehat{\omega}(r)dr}{(1-r)^{1+\tau}}\\
&\leq\int_{|z|}^{1-\frac{1-|z|}{K}}\frac{\widehat{\omega}(r)dr}{(1-r)^{1+\tau}}+\frac{K^{\tau}}{C}\int_{|z|}^1\frac{\widehat{\omega}(r)dr}{(1-r)^{1+\tau}}\\
&\leq\frac{K^{\tau}-1}{\tau}\frac{\widehat{\omega}(z)}{(1-|z|)^{\tau}}+\frac{K^{\tau}}{C}\int_{|z|}^1\frac{\widehat{\omega}(r)dr}{(1-r)^{1+\tau}}.
\end{align*}
That is,
$$\int_{|z|}^1\frac{\widehat{\omega}(r)dr}{(1-r)^{1+\tau}}\leq\frac{C(K^{\tau}-1)}{\tau(C-K^{\tau})}\frac{\widehat{\omega}(z)}{(1-|z|)^{\tau}}$$
and \eqref{hat-integral} is established.

We now denote $\delta_0=\delta_0(\omega)=\frac{1}{2}\min\{\alpha_0(\omega),\gamma_0(\omega),\tau_0(\omega)\}$ and show the desired estimate holds for $s\in(1,1+\delta_0)$ and $\delta\in[-m,\delta_0)$. To this end, note that for any $r\in(0,1)$,
\begin{align}\label{rrr}
       |1-\overline{a}z| &\leq |1-r^2\overline{a}z| + |r^2\overline{a}z - \overline{a}z| \leq |1-r^2\overline{a}z| + (1-r^2)\nonumber \\
        &\leq |1-r^2\overline{a}z| + (1-r^2|a||z|) \leq 2|1-r^2\overline{a}z|.
\end{align}
Since $\gamma\geq m+1+\gamma_0(\omega)$ and $\delta\in[-m,\delta_0)$, we can apply \eqref{rrr} and Lemma \ref{elementary}(c) to obtain that
\begin{align*}
I_2(a,z)&=\int_0^1\frac{(1-r)^{-\delta}\omega(r)dr}{(1-r^2|z|^2)^{\gamma-1}|1-r^2\overline{a}z|^s}\\
&\lesssim\frac{1}{|1-\overline{a}z|^s}\left(\int_0^{|z|}\frac{\omega(r)dr}{(1-r)^{\gamma+\delta-1}}
    +\frac{1}{(1-|z|)^{\gamma-1}}\int_{|z|}^1\frac{\omega(r)dr}{(1-r)^{\delta}}\right)\\
&\lesssim\frac{1}{|1-\overline{a}z|^s}\left(\frac{\widehat{\omega}(z)}{(1-|z|)^{\gamma+\delta-1}}
    +\frac{1}{(1-|z|)^{\gamma-1+\delta-\delta_0}}\int_{|z|}^1\frac{\omega(r)dr}{(1-r)^{\delta_0}}\right).
\end{align*}
Integrating by parts and applying Lemma \ref{elementary-Dcheck} and \eqref{hat-integral}, we obtain that
$$\int_{|z|}^1\frac{\omega(r)dr}{(1-r)^{\delta_0}}=\frac{\widehat{\omega}(z)}{(1-|z|)^{\delta_0}}+\delta_0\int_{|z|}^1\frac{\widehat{\omega}(r)}{(1-|z|)^{1+\delta_0}}dr\lesssim\frac{\widehat{\omega}(z)}{(1-|z|)^{\delta_0}},$$
which gives that
$$I_2(a,z)\lesssim\frac{\widehat{\omega}(z)}{|1-\overline{a}z|^s(1-|z|)^{\gamma+\delta-1}}.$$
It remains to estimate $I_1(a,z)$. Noting that $s+\delta_0-1<\min\{\alpha_0(\omega),\tau_0(\omega)\}$ and $s<\gamma$, integrating by parts and using Lemma \ref{elementary-Dcheck} and \eqref{hat-integral}, we obtain that
\begin{align*}
&\int_{|z|}^1\frac{(1-r)^{-\delta}\omega(r)dr}{(1-r|a|)^{s-1}(1-r|a||z|)^{\gamma-s}}\\
&\leq\frac{1}{(1-|z|)^{\gamma-s+\delta-\delta_0}}\int_{|z|}^1\frac{\omega(r)dr}{(1-r)^{s+\delta_0-1}}\\
&=\frac{1}{(1-|z|)^{\gamma-s+\delta-\delta_0}}\left(\frac{\widehat{\omega}(z)}{(1-|z|)^{s+\delta_0-1}}+(s+\delta_0-1)\int_{|z|}^1\frac{\widehat{\omega}(r)dr}{(1-r)^{s+\delta_0}}\right)\\
&\lesssim\frac{\widehat{\omega}(z)}{(1-|z|)^{\gamma+\delta-1}}.
\end{align*}
This, together with \eqref{rrr} and Lemma \ref{elementary}(c), gives that
\begin{align*}
I_1(a,z)&=\int_0^1\frac{(1-r)^{-\delta}\omega(r)dr}{(1-r^2|a|^2)^{s-1}|1-r^2\overline{a}z|^{\gamma}}\\
&\lesssim\frac{1}{|1-\overline{a}z|^s}\left(\int_0^{|z|}+\int_{|z|}^1\right)\frac{(1-r)^{-\delta}\omega(r)dr}{(1-r|a|)^{s-1}(1-r|a||z|)^{\gamma-s}}\\
&\lesssim\frac{1}{|1-\overline{a}z|^s}\left(\int_0^{|z|}\frac{\omega(r)dr}{(1-r)^{\gamma+\delta-1}}+\frac{\widehat{\omega}(z)}{(1-|z|)^{\gamma+\delta-1}}\right)\\
&\lesssim\frac{\widehat{\omega}(z)}{|1-\overline{a}z|^s(1-|z|)^{\gamma+\delta-1}},
\end{align*}
which finishes the proof.
\end{proof}

\section{Atomic decomposition}\label{atomic-decomposition}

In this section, we are going to establish the atomic decomposition for the analytic tent spaces $AT^{\infty}_q(\omega)$ and $AT^0_q(\omega)$ induced by $\omega\in\DDD$, which will be useful when we deal with the embedding theorems.

Before proceeding, recall that for $r\in(0,1)$, a sequence $Z=\{a_k\}\subset\D$ is said to be an $r$-lattice (in the pseudo-hyperbolic metric) if $\D=\bigcup_{k\geq1}\Delta(a_k,r)$ and $d(a_k,a_j)\geq r/2$ for $k\neq j$. By \cite[Lemma 4.8]{Zh07}, $r$-lattice exists for any $r\in(0,1)$. Moreover, there exists a positive integer $N$ such that for any $r$-lattice $Z=\{a_k\}$ with $r\in(0,1/4)$, each point $z\in\D$ belongs to at most $N$ of the sets $\Delta(a_k,2r)$.

Given an $r$-lattice $Z=\{a_k\}$ and $0<q<\infty$, the sequence tent space $T^{\infty}_q(Z)$ consists of sequences $c=\{c_k\}$ such that
$$\|c\|_{T^{\infty}_q(Z)}^q:
=\esssup_{\xi\in\T}\left(\sup_{a\in\Gamma(\xi)}\frac{1}{1-|a|}\sum_{a_k\in S(a)}|c_k|^q(1-|a_k|)\right)<\infty.$$
Analogously, $T^0_q(Z)$ is the subspace of $T^{\infty}_q(Z)$ consisting of sequences $c=\{c_k\}$ with
$$\lim_{|a|\to1^-}\frac{1}{1-|a|}\sum_{a_k\in S(a)}|c_k|^q(1-|a_k|)=0.$$
As before, we have that a sequence $c=\{c_k\}$ belongs to $T^{\infty}_q(Z)$ (resp. $T^0_q(Z)$) if and only if the measure $d\nu_c:=\sum_{k}|c_k|^q(1-|a_k|)\delta_{a_k}$ is a (resp. vanishing) Carleson measure, where $\delta_{a_k}$ is the Dirac point mass at $a_k$. Moreover, $\|c\|_{T^{\infty}_q(Z)}=\|\nu_c\|_{CM}^{1/q}$.

The main result of this section is as follows, which establishes the atomic decomposition for the analytic tent spaces $AT^{\infty}_q(\omega)$.

\begin{theorem}\label{test-function2}
Let $\omega\in\DDD$, $0<q<\infty$, $\gamma>\big(1+\gamma_0(\omega)\big)\max\left\{1,1/q\right\}$, and let $Z=\{a_k\}$ be an $r$-lattice for some $r\in(0,1/4)$. Then
\begin{enumerate}
    \item [(1)] for any $c=\{c_k\}\in T^{\infty}_q(Z)$, the function
    $$f(z)=\sum_{k\geq1}c_k\frac{(1-|a_k|)^{\gamma}}{\widehat{\omega}(a_k)^{\frac{1}{q}}(1-\overline{a_k}z)^{\gamma}},\quad z\in\D$$
    belongs to $AT^{\infty}_q(\omega)$, and $\|f\|_{AT^{\infty}_q(\omega)}\lesssim\|c\|_{T^{\infty}_q(Z)}$;
    \item [(2)] there exists an $r'$-lattice $Z'=\{z_k\}$ for some sufficiently small $r'\in(0,1/4)$ such that, any $f\in AT^{\infty}_q(\omega)$ has the form
    $$f(z)=\sum_{k\geq1}c_k\frac{(1-|z_k|)^{\gamma}}{\widehat{\omega}(z_k)^{\frac{1}{q}}(1-\overline{z_k}z)^{\gamma}},$$
    where $c=\{c_k\}\in T^{\infty}_q(Z')$ with $\|c\|_{T^{\infty}_q(Z')}\lesssim\|f\|_{AT^{\infty}_q(\omega)}$.
\end{enumerate}
\end{theorem}
\begin{proof}
(1) Fix $a\in\D$. In the case $0<q\leq1$, we have
$$|f(z)|^q\leq\sum_{k\geq1}|c_k|^q\frac{(1-|a_k|)^{q\gamma}}{\widehat{\omega}(a_k)|1-\overline{a_k}z|^{q\gamma}}.$$
Then for sufficiently small $s>0$, Lemma \ref{gFR}(2) with $m=0$ yields that
\begin{align*}
&\int_{\D}\frac{(1-|a|)^s}{|1-\overline{a}z|^{s+1}}|f(z)|^q\omega(z)dA(z)\\
&\ \leq\sum_{k\geq1}|c_k|^q\frac{(1-|a|)^s(1-|a_k|)^{q\gamma}}{\widehat{\omega}(a_k)}
  \int_{\D}\frac{\omega(z)dA(z)}{|1-\overline{a}z|^{s+1}|1-\overline{a_k}z|^{q\gamma}}\\
&\ \lesssim\sum_{k\geq1}\frac{(1-|a|)^s}{|1-\overline{a}a_k|^{s+1}}|c_k|^q(1-|a_k|)\ \lesssim\|c\|^q_{T^{\infty}_q(Z)}.
\end{align*}
Note that for sufficiently small $\epsilon>0$, the sub-harmonic property together with Lemma \ref{FR} gives that
\begin{align*}
\sum_{k\geq1}\frac{(1-|a_k|)^{\gamma-\epsilon}}{|1-\overline{a_k}z|^{\gamma}}
&\lesssim\sum_{k\geq1}(1-|a_k|)^{\gamma-2-\epsilon}\int_{\Delta(a_k,r)}\frac{dA(u)}{|1-\overline{u}z|^{\gamma}}\\
&\lesssim\int_{\D}\frac{(1-|u|)^{\gamma-2-\epsilon}}{|1-\overline{u}z|^{\gamma}}dA(u)\\
&\asymp(1-|z|)^{-\epsilon}.
\end{align*}
Hence in the case $q>1$, we can use H\"older's inequality to obtain that for sufficiently small $\epsilon>0$,
\begin{align*}
|f(z)|^q&\leq\left(\sum_{k\geq1}\frac{(1-|a_k|)^{\gamma-\epsilon}}{|1-\overline{a_k}z|^{\gamma}}\right)^{\frac{q}{q'}}\cdot
    \sum_{k\geq1}|c_k|^q\frac{(1-|a_k|)^{\gamma+(q-1)\epsilon}}{\widehat{\omega}(a_k)|1-\overline{a_k}z|^{\gamma}}\\
&\lesssim(1-|z|)^{-(q-1)\epsilon}\sum_{k\geq1}|c_k|^q\frac{(1-|a_k|)^{\gamma+(q-1)\epsilon}}{\widehat{\omega}(a_k)|1-\overline{a_k}z|^{\gamma}},
\end{align*}
which, combined with Lemma \ref{gFR}(2), gives that
\begin{align*}
&\int_{\D}\frac{(1-|a|)^s}{|1-\overline{a}z|^{s+1}}|f(z)|^q\omega(z)dA(z)\\
&\ \lesssim\sum_{k\geq1}|c_k|^q\frac{(1-|a|)^s(1-|a_k|)^{\gamma+(q-1)\epsilon}}{\widehat{\omega}(a_k)}
    \int_{\D}\frac{(1-|z|)^{-(q-1)\epsilon}\omega(z)dA(z)}{|1-\overline{a}z|^{s+1}|1-\overline{a_k}z|^{\gamma}}\\
&\ \lesssim\sum_{k\geq1}\frac{(1-|a|)^s}{|1-\overline{a}a_k|^{s+1}}|c_k|^q(1-|a_k|)\ \lesssim\|c\|^q_{T^{\infty}_q(Z)}.
\end{align*}
Since $a\in\D$ is arbitrary, in both cases, we have $f\in AT^{\infty}_q(\omega)$, and $\|f\|_{AT^{\infty}_q(\omega)}\lesssim\|c\|_{T^{\infty}_q(Z)}$.

(2) Let $r'$ be small enough in the sense that $r'/r$ is small. Then we can construct an $r'$-lattice $Z'=\{a_{kj}\}_{1\leq k<\infty,1\leq j\leq J}$ and a disjoint decomposition
$$\D=\bigcup_{k=1}^{\infty}\bigcup_{j=1}^JD_{kj}$$
of $\D$ such that $a_{kj}\in\Delta(a_k,r)$ and $D_{kj}\subset\Delta(a_k,r)$ for any $k\geq1$ and $1\leq j\leq J$; see \cite[p. 64]{Zh05} for the details.

Define an operator $S$ on $\mathcal{H}(\D)$ by
$$Sf(z)=\sum_{k=1}^{\infty}\sum_{j=1}^J\frac{f(a_{kj})}{(1-\overline{a_{kj}}z)^{\gamma}}
\int_{D_{kj}}(1-|z|)^{\gamma-2}dA(z).$$
By \cite[Lemma 2.29]{Zh05} and \eqref{D-constant}, there exist constants $C_1,C_2>0$, independent of $r'$ and $r$, such that
\begin{align*}
&|f(z)-Sf(z)|\\
&\ \leq C_1\eta\sum_{k=1}^{\infty}\frac{(1-|a_k|)^{\gamma-\frac{2}{q}}}{|1-\overline{a_k}z|^{\gamma}}
    \left(\int_{\Delta(a_k,2r)}|f(u)|^qdA(u)\right)^{\frac{1}{q}}\\
&\ \leq C_2\eta\sum_{k=1}^{\infty}\frac{(1-|a_k|)^{\gamma}}{\widehat{\omega}(a_k)^{\frac{1}{q}}
    |1-\overline{a_k}z|^{\gamma}}
    \left(\int_{\Delta(a_k,2r)}|f(u)|^q\frac{\widehat{\omega}(u)}{(1-|u|)^2}dA(u)\right)^{\frac{1}{q}}
\end{align*}
for all $z\in\D$ and $f\in \mathcal{H}(\D)$, where $\eta=\tanh^{-1}r'+r'r^{2n(1-\frac{1}{q})-1}$. Then by (1) and Lemma \ref{hat-equivalent}, for any $s>0$,
\begin{align*}
&\|f-Sf\|_{AT^{\infty}_q(\omega)}\\
&\leq C_3\eta\left(\sup_{a\in\D}\sum_{k=1}^{\infty}\frac{(1-|a|)^s(1-|a_k|)}{|1-\overline{a}a_k|^{s+1}}
    \int_{\Delta(a_k,2r)}|f(u)|^q\frac{\widehat{\omega}(u)}{(1-|u|)^2}dA(u)\right)^{1/q}\\
&\leq C_4N\eta\left(\sup_{a\in\D}\int_{\D}\frac{(1-|a|)^s}{|1-\overline{a}u|^{s+1}}|f(u)|^q
    \frac{\widehat{\omega}(u)}{1-|u|}dA(u)\right)^{1/q}\\
&\leq C_5N\eta\|f\|_{AT^{\infty}_q(\omega)},
\end{align*}
where $C_3,C_4,C_5$ are all independent of $r'$ and $r$. Hence we can choose $r'$ small enough so that $C_5N\eta<1$ and consequently, the operator $S$ is invertible on $AT^{\infty}_q(\omega)$. Then for any $f\in AT^{\infty}_q(\omega)$, there exists $g\in AT^{\infty}_q(\omega)$ such that
$$f(z)=Sg(z)=\sum_{k=1}^{\infty}\sum_{j=1}^Jc_{kj}\frac{(1-|a_{kj}|)^{\gamma}}{\widehat{\omega}(a_{kj})^{\frac{1}{q}}(1-\overline{a_{kj}}z)^{\gamma}},$$
where
$$c_{kj}=g(a_{kj})\widehat{\omega}(a_{kj})^{\frac{1}{q}}(1-|a_{kj}|)^{-\gamma}\int_{D_{kj}}(1-|z|)^{\gamma-2}dA(z).$$
It remains to show that the sequence $c=\{c_{kj}\}$ belongs to $T^{\infty}_q(Z')$. In fact, since $D_{kj}\subset\Delta(a_k,r)$ and $a_{kj}\in\Delta(a_k,r)$, we deduce from the sub-harmonic property of $|g|^q$ and \eqref{D-constant} that
$$|c_{kj}|^q\lesssim\frac{1}{1-|a_{kj}|}\int_{\Delta(a_k,2r)}|g(u)|^q\frac{\widehat{\omega}(u)}{1-|u|}dA(u),$$
which, together with Lemma \ref{hat-equivalent}, implies that for any $s>0$ and $a\in\D$,
\begin{align*}
&\sum_{k=1}^{\infty}\sum_{j=1}^J\frac{(1-|a|)^s}{|1-\overline{a}a_{kj}|^{s+1}}|c_{kj}|^q(1-|a_{kj}|)\\
&\ \lesssim J\sum_{k=1}^{\infty}\frac{(1-|a|)^s}{|1-\overline{a}a_k|^{s+1}}\int_{\Delta(a_k,2r)}|g(u)|^q
    \frac{\widehat{\omega}(u)}{1-|u|}dA(u)\\
&\ \lesssim NJ\int_{\D}\frac{(1-|a|)^s}{|1-\overline{a}u|^{s+1}}|g(u)|^q\frac{\widehat{\omega}(u)}{1-|u|}dA(u)\\
&\ \lesssim NJ\|g\|^q_{AT^{\infty}_q(\omega)}\\
&\ \leq NJ\|S^{-1}\|^q\|f\|^q_{AT^{\infty}_q(\omega)}.
\end{align*}
Therefore, $c=\{c_{kj}\}\in T^{\infty}_q(Z')$ and $\|c\|_{T^{\infty}_q(Z')}\lesssim\|f\|_{AT^{\infty}_q(\omega)}$. The proof is complete.
\end{proof}

Using the same method, we can also establish the atomic decomposition for the spaces $AT^0_q(\omega)$.

\begin{theorem}\label{atom2}
Let $\omega\in\DDD$, $0<q<\infty$, $\gamma>\big(1+\gamma_0(\omega)\big)\max\left\{1,1/q\right\}$. Then there exists an $r$-lattice $Z=\{a_k\}$ such that the space $AT^0_q(\omega)$ consists exactly of functions of the form
    $$f(z)=\sum_{k\geq1}c_k\frac{(1-|a_k|)^{\gamma}}{\widehat{\omega}(a_k)^{\frac{1}{q}}(1-\overline{a_k}z)^{\gamma}},\quad z\in\D,$$
where $\{c_k\}\in T^0_q(Z)$.
\end{theorem}

As an application of Theorem \ref{test-function2}, we can determine the behavior of the dilation of functions in $AT^{\infty}_q(\omega)$. For an analytic function $f$ on $\D$ and $\rho\in(0,1)$, the dilation $f_{\rho}$ is defined by $f_{\rho}(z):=f(\rho z)$ for $z\in\D$.

\begin{corollary}\label{dilation}
Let $0<q<\infty$, $\omega\in\DDD$ and $f\in AT^{\infty}_q(\omega)$. Then for any $\rho\in(0,1)$, $\|f_{\rho}\|_{AT^{\infty}_q(\omega)}\lesssim\|f\|_{AT^{\infty}_q(\omega)}$.
\end{corollary}
\begin{proof}
Fix $\gamma>\big(1+\beta_0(\omega)+\gamma_0(\omega)\big)\max\{1,1/q\}$. Then by Theorem \ref{test-function2}, there exist an $r$-lattice $Z=\{z_k\}$ for some sufficiently small $r\in(0,1/4)$ and a sequence $c=\{c_k\}\in T^{\infty}_q(Z)$ such that
$$f(z)=\sum_{k\geq1}c_k\frac{(1-|z_k|)^{\gamma}}{\widehat{\omega}(z_k)^{\frac{1}{q}}(1-\overline{z_k}z)^{\gamma}}$$
and $\|c\|_{T^{\infty}_q(Z)}\asymp\|f\|_{AT^{\infty}_q(\omega)}$. Consequently,
$$f_{\rho}(z)=\sum_{k\geq1}c_k\frac{(1-|z_k|)^{\gamma}}{\widehat{\omega}(z_k)^{\frac{1}{q}}(1-\overline{z_k}\rho z)^{\gamma}}.$$
In the case $0<q\leq1$, we can use Lemma \ref{gFR} similarly as in the proof of Theorem \ref{test-function2}(1) to obtain that for any $a\in\D$ and sufficiently small $s>0$,
\begin{align*}
&\int_{\D}\frac{(1-|a|)^s}{|1-\overline{a}z|^{s+1}}|f_{\rho}(z)|^q\omega(z)dA(z)\\
&\ \leq\sum_{k\geq1}|c_k|^q\frac{(1-|a|)^s(1-|z_k|)^{q\gamma}}{\widehat{\omega}(z_k)}
  \int_{\D}\frac{\omega(z)dA(z)}{|1-\overline{a}z|^{s+1}|1-\overline{\rho z_k}z|^{q\gamma}}\\
&\ \lesssim\sum_{k\geq1}\frac{(1-|a|)^s}{|1-\overline{\rho a}z_k|^{s+1}}|c_k|^q
  \frac{(1-|z_k|)^{q\gamma}}{(1-\rho|z_k|)^{q\gamma-1}}\frac{\widehat{\omega}(\rho z_k)}{\widehat{\omega}(z_k)},
\end{align*}
which, combined with Lemma \ref{elementary}(b) and the fact that $q\gamma-\beta_0(\omega)-1>0$, implies that
\begin{align*}
&\int_{\D}\frac{(1-|a|)^s}{|1-\overline{a}z|^{s+1}}|f_{\rho}(z)|^q\omega(z)dA(z)\\
&\ \lesssim\sum_{k\geq1}\frac{(1-|a|)^s}{|1-\overline{\rho a}z_k|^{s+1}}|c_k|^q
  \frac{(1-|z_k|)^{q\gamma-\beta_0(\omega)}}{(1-\rho|z_k|)^{q\gamma-\beta_0(\omega)-1}}\\
&\ \leq\sum_{k\geq1}\frac{(1-|\rho a|)^s}{|1-\overline{\rho a}z_k|^{s+1}}|c_k|^q(1-|z_k|).
\end{align*}
Therefore, $\|f_{\rho}\|_{AT^{\infty}_q(\omega)}\lesssim\|c\|_{T^{\infty}_q(Z)}\asymp\|f\|_{AT^{\infty}_q(\omega)}$. The case $q>1$ is similar and we omit the details.
\end{proof}

\section{Embedding theorems}\label{embedding-theorems}

The purpose of this section is to characterize the bounded and compact embeddings $I_d:AT^{\infty}_p(\omega)\to T^{\infty}_q(\mu)$. To this end, we need two auxiliary results. The first one establishes the growth estimates for the derivatives of functions in $AT^{\infty}_q(\omega)$, and the second one gives a family of test functions in $AT^{\infty}_q(\omega)$.

\begin{lemma}\label{thm:growth-estimate-unit-disk}
Let $0 < q < \infty$, $\omega \in \mathcal{D}$, $m\in\mathbb{N}_0$ and $f \in AT_{q}^{\infty}(\omega)$. Then for any $z\in\D$,
$$
|f^{(m)}(z)| \lesssim \frac{\|f\|_{AT_{q}^{\infty}(\omega)}}{\widehat{\omega}(z)^{1/q}(1-|z|)^m}.
$$
\end{lemma}
\begin{proof}
Fix $s>0$. By the sub-harmonic property of $|f|^q$, \eqref{D-constant} and Lemma \ref{hat-equivalent}, we have
\begin{align*}
|f(z)|^q &\lesssim \frac{1}{(1-|z|)^2} \int_{\Delta(z,r)} |f(u)|^q  dA(u)\\
&\lesssim \frac{1}{\widehat{\omega}(z)(1-|z|)} \int_{\Delta(z,r)} |f(u)|^q
  \frac{\widehat{\omega}(u)}{1-|u|}dA(u)\\
&\asymp\frac{1}{\widehat{\omega}(z)}\int_{\Delta(z,r)}
  \frac{(1-|z|)^{s}}{|1-\overline{z}u|^{s+1}}|f(u)|^q\frac{\widehat{\omega}(u)}{1-|u|}dA(u)\\
&\le\frac{1}{\widehat{\omega}(z)}\int_{\D}
  \frac{(1-|z|)^{s}}{|1-\overline{z}u|^{s+1}}|f(u)|^q\frac{\widehat{\omega}(u)}{1-|u|}dA(u)\\
&\asymp\frac{1}{\widehat{\omega}(z)} \int_{\D}
  \frac{(1-|z|)^{s}}{|1-\overline{z}u|^{s+1}}|f(u)|^q\omega(u)dA(u)\\
&\lesssim\frac{1}{\widehat{\omega}(z)}\|f\|^q_{AT^{\infty}_q(\omega)},
\end{align*}
which proves the case $m=0$. The case $m\geq1$ then follows directly from the Cauchy integral formula. In fact, let $\Upsilon$ be the circle centered at $z$ with radius $\frac{1}{2}(1-|z|)$. Then
\begin{align*}
|f^{(m)}(z)|&\leq\frac{m!}{2\pi}\int_{\Upsilon}\frac{|f(\xi)|}{|\xi-z|^{m+1}}|d\xi|\leq\frac{2^mm!}{(1-|z|)^m}\sup_{\xi\in\Upsilon}|f(\xi)|\\
&\lesssim\frac{\|f\|_{AT^{\infty}_q(\omega)}}{(1-|z|)^m}\sup_{\xi\in\Upsilon}\frac{1}{\widehat{\omega}(\xi)^{1/q}}\leq\frac{\|f\|_{AT^{\infty}_q(\omega)}}{\widehat{\omega}\left(\frac{1+|z|}{2}\right)^{1/q}(1-|z|)^m}\\
&\lesssim\frac{\|f\|_{AT^{\infty}_q(\omega)}}{\widehat{\omega}(z)^{1/q}(1-|z|)^m}.
\end{align*}
The proof is complete.
\end{proof}

\begin{lemma}\label{test-function}
Let $0<q<\infty$, $\omega\in\DD$, and let $\gamma>\max\left\{\frac{\gamma_0(\omega)}{q},\frac{1}{q}\right\}$. Then for any $z\in\D$, the function
$$f_z(u)=\frac{1}{\widehat{\omega}(z)^{1/q}}\left(\frac{1-|z|}{1-\overline{z}u}\right)^{\gamma},\quad u\in\D$$
belongs to $AT^{\infty}_q(\omega)$, and $\|f_z\|_{AT^{\infty}_q(\omega)}\lesssim1$.
\end{lemma}
\begin{proof}
Fix $s>0$. Then for any $a\in\D$, Lemma \ref{gFR}(1) gives that
\begin{align*}
&\int_{\D}\frac{(1-|a|)^s}{|1-\overline{a}u|^{s+1}}|f_z(u)|^q\omega(u)dA(u)\\
&\ =\frac{(1-|a|)^s(1-|z|)^{q\gamma}}{\widehat{\omega}(z)}\int_{\D}\frac{\omega(u)dA(u)}{|1-\overline{a}u|^{s+1}|1-\overline{z}u|^{q\gamma}}\lesssim1.
\end{align*}
Since $a\in\D$ is arbitrary, we obtain that $f_z\in AT^{\infty}_q(\omega)$, and $\|f_z\|_{AT^{\infty}_q(\omega)}\lesssim1$.
\end{proof}

We are now ready to characterize the boundedness of $I_d:AT^{\infty}_p(\omega)\to T^{\infty}_q(\mu)$ and $I_d:AT^0_p(\omega)\to T^0_q(\mu)$. Recall that for a positive Borel measure $\mu$ on $\D$ and $r\in(0,1)$,
$$G_{\mu,r}(z)=\frac{\mu(\Delta(z,r))^{\frac{1}{q}}}{\widehat{\omega}(z)^{\frac{1}{p}}(1-|z|)^{\frac{1}{q}}}.$$

\begin{theorem}\label{bdd-em-1}
Let $0<p\leq q<\infty$, $\omega\in\DDD$, and let $\mu$ be a positive Borel measure on $\D$. Then the following conditions are equivalent:
\begin{enumerate}
    \item [(a)] $I_d:AT^{\infty}_p(\omega)\to T^{\infty}_q(\mu)$ is bounded;
    \item [(b)] $I_d:AT^0_p(\omega)\to T^0_q(\mu)$ is bounded;
    \item [(c)] for some (or any) $r\in(0,1/4)$, $G_{\mu,r}\in L^{\infty}$.
\end{enumerate}
Moreover, in these cases,
$$\|I_d\|\asymp\left\|G_{\mu,r}\right\|_{L^{\infty}}.$$
\end{theorem}
\begin{proof}
Fix $s>0$. Suppose first that (a) holds. Then, since the constant function $1$ belongs to $AT^{\infty}_{p}(\omega)$, we have
$$\|I_d\|^q\gtrsim\|1\|^q_{T^{\infty}_q(\mu)}
\gtrsim\sup_{a\in\D}\int_{\D}\frac{(1-|a|)^s}{|1-\overline{a}z|^{s+1}}d\mu(z)\geq\mu(\D),$$
which implies that for any $r\in(0,1/4)$,
$$\sup_{|z|\leq\frac{1}{2}}G_{\mu,r}(z)^q\lesssim\mu(\D)\lesssim\|I_d\|^q.$$
For any $z\in\D$ with $|z|>1/2$, write
$$\tilde{z}=\big(1-2(1-|z|)\big)\frac{z}{|z|}.$$
Then it is easy to verify that $\Delta(z,r)\subset S(\tilde{z})$ whenever $r\in(0,1/4)$. Let the function $f_z$ be defined as in Lemma \ref{test-function}. Noting that $|1-\overline{z}u|\asymp1-|z|$ for any $u\in S(\tilde{z})$, we have
\begin{align}\label{bdd-nece}
\|I_d\|^q
&\gtrsim\|f_z\|^q_{T^{\infty}_q(\mu)}\geq\frac{1}{1-|\tilde{z}|}\int_{S(\tilde{z})}|f_z(u)|^qd\mu(u)\nonumber\\
&=\frac{(1-|z|)^{q\gamma}}{\widehat{\omega}(z)^{\frac{q}{p}}(1-|\tilde{z}|)}\int_{S(\tilde{z})}\frac{d\mu(u)}{|1-\overline{z}u|^{q\gamma}}\gtrsim G_{\mu,r}(z)^q.
\end{align}
Therefore, (c) holds and $\left\|G_{\mu,r}\right\|_{L^{\infty}}\lesssim\|I_d\|$.

If (b) holds, then (c) also holds since it is clear that $f_z\in H^{\infty}\subset AT^0_p(\omega)$ for each $z\in\D$, where $H^{\infty}$ is the algebra of bounded analytic functions on $\D$.

Suppose next that (c) holds, i.e. for some $r\in(0,1/4)$, $G_{\mu,r}\in L^{\infty}$. We first establish (a). Let $\{a_k\}\subset\D$ be an $r$-lattice. Then for any $f\in AT^{\infty}_p(\omega)$ and $a\in\D$, the sub-harmonic property of $|f|^p$ together with Lemmas \ref{thm:growth-estimate-unit-disk}, \ref{hat-equivalent} and \eqref{D-constant} gives that
\begin{align}\label{suff-embedding}
&\int_{\D}\frac{(1-|a|)^s}{|1-\overline{a}z|^{s+1}}|f(z)|^qd\mu(z)\nonumber\\
&\leq\sum_{k\geq1}\int_{\Delta(a_k,r)}\frac{(1-|a|)^s}{|1-\overline{a}z|^{s+1}}|f(z)|^qd\mu(z)\nonumber\\
&\lesssim\sum_{k\geq1}\frac{(1-|a|)^s}{|1-\overline{a}a_k|^{s+1}}\int_{\Delta(a_k,r)}\frac{|f(z)|^{q-p}}{(1-|z|)^2}
   \int_{\Delta(z,r)}|f(u)|^pdA(u)d\mu(z)\nonumber\\
&\lesssim\|f\|^{q-p}_{AT^{\infty}_p(\omega)}\sum_{k\geq1}\frac{(1-|a|)^s}{|1-\overline{a}a_k|^{s+1}} 
    \int_{\Delta(a_k,r)}\frac{d\mu(z)}{\widehat{\omega}(z)^{\frac{q-p}{p}}(1-|z|)^2}\int_{\Delta(a_k,2r)}|f(u)|^pdA(u)
    \nonumber\\
&\asymp\|f\|^{q-p}_{AT^{\infty}_p(\omega)}\sum_{k\geq1}\frac{(1-|a|)^s}{|1-\overline{a}a_k|^{s+1}}   
   \frac{\mu(\Delta(a_k,r))}{\widehat{\omega}(a_k)^{\frac{q-p}{p}}(1-|a_k|)^2}\int_{\Delta(a_k,2r)}|f(u)|^pdA(u)
   \nonumber\\
&\leq\left\|G_{\mu,r}\right\|_{L^{\infty}}^q\|f\|^{q-p}_{AT^{\infty}_p(\omega)}\sum_{k\geq1}\frac{(1-|a|)^s}{|1-\overline{a}a_k|^{s+1}}
  \frac{\widehat{\omega}(a_k)}{1-|a_k|}\int_{\Delta(a_k,2r)}|f(u)|^pdA(u)\nonumber\\
&\lesssim\left\|G_{\mu,r}\right\|_{L^{\infty}}^q\|f\|^{q-p}_{AT^{\infty}_p(\omega)}\int_{\D}\frac{(1-|a|)^s}{|1-\overline{a}u|^{s+1}}|f(u)|^p
   \frac{\widehat{\omega}(u)}{1-|u|}dA(u)\nonumber\\
&\asymp\left\|G_{\mu,r}\right\|_{L^{\infty}}^q\|f\|^{q-p}_{AT^{\infty}_p(\omega)}\int_{\D}\frac{(1-|a|)^s}{|1-\overline{a}u|^{s+1}}|f(u)|^p\omega(u)dA(u).
\end{align}
Since $a\in\D$ and $f\in AT^{\infty}_p(\omega)$ are both arbitrary, we conclude (a) with the norm estimate $\|I_d\|\lesssim \left\|G_{\mu,r}\right\|_{L^{\infty}}$. To establish (b), it is enough to show that $AT^0_p(\omega)\subset T^0_q(\mu)$, which also follows from \eqref{suff-embedding}. The proof is complete.
\end{proof}

\begin{theorem}\label{bdd-em-2}
Let $0<q<p<\infty$, $\omega\in\DDD$, and let $\mu$ be a positive Borel measure on $\D$. Then the following conditions are equivalent:
\begin{enumerate}
    \item [(a)] $I_d:AT^{\infty}_p(\omega)\to T^{\infty}_q(\mu)$ is bounded;
    \item [(b)] $I_d:AT^0_p(\omega)\to T^0_q(\mu)$ is bounded;
    \item [(c)] for some (or any) $r\in(0,1/4)$, $G_{\mu,r}\in T^{\infty}_{\frac{pq}{p-q},-1}$.
\end{enumerate}
Moreover, in these cases,
$$\|I_d\|\asymp\left\|G_{\mu,r}\right\|_{T^{\infty}_{\frac{pq}{p-q},-1}}.$$
\end{theorem}
\begin{proof}
Fix $s>0$. Suppose first that (c) holds, i.e. for some $r\in(0,1/4)$, $G_{\mu,r}\in T^{\infty}_{\frac{pq}{p-q},-1}$. Then for any $a\in\D$ and $f\in AT^{\infty}_p(\omega)$, the sub-harmonic property of $|f|^q$ together with Fubini's theorem, H\"older's inequality and Lemma \ref{hat-equivalent} gives that
\begin{align}\label{suff-embedding2}
&\int_{\D}\frac{(1-|a|)^s}{|1-\overline{a}z|^{s+1}}|f(z)|^qd\mu(z)\nonumber\\
&\ \lesssim\int_{\D}\frac{(1-|a|)^s}{|1-\overline{a}z|^{s+1}}\frac{1}{(1-|z|)^2}\int_{\Delta(z,r)}|f(u)|^qdA(u)d\mu(z)
  \nonumber\\
&\ \asymp\int_{\D}\frac{(1-|a|)^s}{|1-\overline{a}u|^{s+1}}|f(u)|^q\frac{\mu(\Delta(u,r))}{(1-|u|)^2}dA(u)\nonumber\\
&\ =\int_{\D}\frac{(1-|a|)^s}{|1-\overline{a}u|^{s+1}}|f(u)|^q
    \frac{\widehat{\omega}(u)^{\frac{q}{p}}}{(1-|u|)^{\frac{q}{p}}}G_{\mu,r}(u)^q(1-|u|)^{\frac{q}{p}-1}dA(u)\nonumber\\
&\ \leq\left(\int_{\D}\frac{(1-|a|)^s}{|1-\overline{a}u|^{s+1}}|f(u)|^p
    \frac{\widehat{\omega}(u)}{1-|u|}dA(u)\right)^{\frac{q}{p}}\cdot\nonumber\\
&\ \qquad\qquad    \left(\int_{\D}\frac{(1-|a|)^s}{|1-\overline{a}u|^{s+1}}G_{\mu,r}(u)^{\frac{pq}{p-q}}
    \frac{dA(u)}{1-|u|}\right)^{\frac{p-q}{p}}\nonumber\\
&\ \lesssim\left(\int_{\D}\frac{(1-|a|)^s}{|1-\overline{a}u|^{s+1}}|f(u)|^p\omega(u)dA(u)\right)^{\frac{q}{p}}
    \cdot\left\|G_{\mu,r}\right\|^q_{T^{\infty}_{\frac{pq}{p-q},-1}}\\
&\ \lesssim\|f\|^q_{AT^{\infty}_{p}(\omega)}\cdot\left\|G_{\mu,r}\right\|^q_{T^{\infty}_{\frac{pq}{p-q},-1}}.\nonumber
\end{align}
Since $a\in\D$ is arbitrary, we obtain that
$$\|f\|_{T^{\infty}_q(\mu)}\lesssim\|f\|_{AT^{\infty}_{p}(\omega)}\cdot\left\|G_{\mu,r}\right\|_{T^{\infty}_{\frac{pq}{p-q},-1}}.$$
Hence (a) holds, and
$$\|I_d\|\lesssim
\left\|G_{\mu,r}\right\|_{T^{\infty}_{\frac{pq}{p-q},-1}}.$$
To establish (b), it is sufficient to note that \eqref{suff-embedding2} implies $AT^0_p(\omega)\subset T^0_q(\mu)$.

Suppose next that (a) holds. Fix $r\in(0,1/4)$ and let 
$Z=\{a_k\}\subset\D$ be an $r$-lattice. For any $c=\{c_k\}\in T^{\infty}_p(Z)$, define
\begin{equation}\label{ft}
f_t(z)=\sum_{k\geq1}c_kr_k(t)\frac{(1-|a_k|)^{\gamma}}{\widehat{\omega}(a_k)^{\frac{1}{p}}(1-\overline{a_k}z)^{\gamma}},\quad z\in\D,
\end{equation}
where $\gamma>\big(1+\gamma_0(\omega)\big)\max\{1,1/p\}$, and $\{r_k\}$ is the sequence of Rademacher functions on $[0,1]$ (see for instance \cite[Appendix A]{Du70}). Then by Theorem \ref{test-function2}(1), $f_t\in AT^{\infty}_p(\omega)$ and $\|f_t\|_{AT^{\infty}_p(\omega)}\lesssim\|c\|_{T^{\infty}_p(Z)}$ for almost every $t\in[0,1]$. Since $I_d:AT^{\infty}_p(\omega)\to T^{\infty}_q(\mu)$ is bounded, we have for any $a\in\D$ and almost every $t\in[0,1]$,
$$\int_{\D}\frac{(1-|a|)^s}{|1-\overline{a}z|^{s+1}}\left|\sum_{k\geq1}c_kr_k(t)\frac{(1-|a_k|)^{\gamma}}{\widehat{\omega}(a_k)^{\frac{1}{p}}(1-\overline{a_k}z)^{\gamma}}\right|^qd\mu(z)
\lesssim\|I_d\|^q\|c\|^q_{T^{\infty}_p(Z)}.$$
Integrating with respect to $t$ on $[0,1]$ and using Fubini's theorem and Khinchine's inequality (see for instance \cite[Appendix A]{Du70}), we arrive at
\begin{equation}\label{F-K}
\int_{\D}\frac{(1-|a|)^s}{|1-\overline{a}z|^{s+1}}\left(\sum_{k\geq1}|c_k|^2
\frac{(1-|a_k|)^{2\gamma}}{\widehat{\omega}(a_k)^{\frac{2}{p}}|1-\overline{a_k}z|^{2\gamma}}\right)^{\frac{q}{2}}d\mu(z)
\lesssim\|I_d\|^q\|c\|^q_{T^{\infty}_p(Z)}.
\end{equation}
Since each point $z\in\D$ belongs to at most $N$ of the sets $\Delta(a_k,2r)$, we get
\begin{align*}
\left(\sum_{k\geq1}|c_k|^2\frac{(1-|a_k|)^{2\gamma}}{\widehat{\omega}(a_k)^{\frac{2}{p}}|1-\overline{a_k}z|^{2\gamma}}
    \right)^{\frac{q}{2}}
&\geq\left(\sum_{k\geq1}|c_k|^2\frac{(1-|a_k|)^{2\gamma}\chi_{\Delta(a_k,2r)}(z)}
    {\widehat{\omega}(a_k)^{\frac{2}{p}}|1-\overline{a_k}z|^{2\gamma}}\right)^{\frac{q}{2}}\\
&\gtrsim\sum_{k\geq1}|c_k|^q\frac{\chi_{\Delta(a_k,2r)}(z)}{\widehat{\omega}(a_k)^{\frac{q}{p}}}.   
\end{align*}
Inserting this into \eqref{F-K} yields that
$$\int_{\D}\frac{(1-|a|)^s}{|1-\overline{a}z|^{s+1}}
\sum_{k\geq1}|c_k|^q\frac{\chi_{\Delta(a_k,2r)}(z)}{\widehat{\omega}(a_k)^{\frac{q}{p}}}d\mu(z)
\lesssim\|I_d\|^q\|c\|^q_{T^{\infty}_p(Z)},$$
which implies that
$$\sum_{k\geq1}\frac{(1-|a|)^s}{|1-\overline{a}a_k|^{s+1}}|c_k|^qG_{\mu,2r}(a_k)^q(1-|a_k|)
\lesssim\|I_d\|^q\|c\|^q_{T^{\infty}_p(Z)}.$$
Since $a\in\D$ is arbitrary, we obtain that $\left\{c_kG_{\mu,2r}(a_k)\right\}\in T^{\infty}_q(Z)$, and
$$\left\|\left\{c_kG_{\mu,2r}(a_k)\right\}\right\|_{T^{\infty}_q(Z)}\lesssim\|I_d\|\|c\|_{T^{\infty}_p(Z)}.$$
The arbitrariness of $c\in T^{\infty}_p(Z)$ together with \cite[Lemma 2.2]{CW22} then gives that $\{G_{\mu,2r}(a_k)\}\in T^{\infty}_{\frac{pq}{p-q}}(Z)$ with $\|\{G_{\mu,2r}(a_k)\}\|_{T^{\infty}_{\frac{pq}{p-q}}(Z)}\lesssim\|I_d\|$. Consequently, for any $a\in\D$,
\begin{align}\label{G-inequality}
&\int_{\D}\frac{(1-|a|)^s}{|1-\overline{a}z|^{s+1}}G_{\mu,r}(z)^{\frac{pq}{p-q}}\frac{dA(z)}{1-|z|}\nonumber\\
&\ \leq\sum_{k\geq1}\int_{\Delta(a_k,r)}\frac{(1-|a|)^s}{|1-\overline{a}z|^{s+1}}
    G_{\mu,r}(z)^{\frac{pq}{p-q}}\frac{dA(z)}{1-|z|}\nonumber\\
&\ \lesssim\sum_{k\geq1}\frac{(1-|a|)^s}{|1-\overline{a}a_k|^{s+1}}
    G_{\mu,2r}(a_k)^{\frac{pq}{p-q}}(1-|a_k|)\\
&\ \lesssim\|\{G_{\mu,2r}(a_k)\}\|_{T^{\infty}_{\frac{pq}{p-q}}(Z)}^{\frac{pq}{p-q}}\ \lesssim\|I_d\|^{\frac{pq}{p-q}}.
    \nonumber
\end{align}
Therefore, $G_{\mu,r}\in T^{\infty}_{\frac{pq}{p-q},-1}$ and $\|G_{\mu,r}\|_{T^{\infty}_{\frac{pq}{p-q},-1}}\lesssim\|I_d\|$. That is, (a) implies (c).

If (b) holds, then define $f_t$ as \eqref{ft} for $c=\{c_k\}\in T^0_p(Z)$. Then by Theorem \ref{atom2}, $f_t\in AT^0_p(\omega)$ for almost every $t\in[0,1]$. Since $I_d:AT^0_p(\omega)\to T^0_q(\mu)$ is bounded, for almost every $t\in[0,1]$,
$$\lim_{|a|\to1^-}\int_{\D}\frac{(1-|a|)^s}{|1-\overline{a}z|^{s+1}}|f_t(z)|^qd\mu(z)=0.$$
Moreover, for every $a\in\D$ and almost every $t\in[0,1]$,
$$\int_{\D}\frac{(1-|a|)^s}{|1-\overline{a}z|^{s+1}}|f_t(z)|^qd\mu(z)\lesssim\|f_t\|^q_{T^{\infty}_q(\mu)}
\lesssim\|I_d\|^q\|f_t\|^q_{AT^{\infty}_p(\omega)}\lesssim\|I_d\|^q\|c\|^q_{T^{\infty}_p(Z)}.$$
Hence by the Lebesgue dominated convergence theorem,
$$\lim_{|a|\to1^-}\int_0^1\int_{\D}\frac{(1-|a|)^s}{|1-\overline{a}z|^{s+1}}|f_t(z)|^qd\mu(z)dt=0.$$
Then, by the same process as before, we obtain
$$\lim_{|a|\to1^-}\sum_{k\geq1}\frac{(1-|a|)^s}{|1-\overline{a}a_k|^{s+1}}|c_k|^qG_{\mu,2r}(a_k)^q(1-|a_k|)=0,$$
which indicates that $\{c_kG_{\mu,2r}(a_k)\}\in T^0_q(Z)$. The arbitrariness of $c\in T^0_p(Z)$ together with \cite[Lemma 2.2]{CW22} gives that $\{G_{\mu,2r}(a_k)\}\in T^{\infty}_{\frac{pq}{p-q}}(Z)$. Therefore, (c) also holds and the proof is complete.
\end{proof}

Putting Theorems \ref{bdd-em-1} and \ref{bdd-em-2} together, we establish Theorem \ref{bounded-embedding}.

In the rest part of this section, we investigate the compactness of $I_d:AT^{\infty}_p(\omega)\to T^{\infty}_q(\mu)$ and $I_d:AT^0_p(\omega)\to T^0_q(\mu)$. Before proceeding, note that for any $0<q<\infty$ and positive Borel measure $\mu$ on $\D$, the tent space $T^{\infty}_q(\mu)$ is embedded into $L^q(\mu)$ continuously. Combining this fact with Lemma \ref{thm:growth-estimate-unit-disk}, and applying Montel's theorem, we can obtain the following characterization for the compactness of the embedding $I_d:AT^{\infty}_p(\omega)\to T^{\infty}_q(\mu)$.

\begin{lemma}\label{cpt-ele}
Let $0<p,q<\infty$, $\omega\in\DDD$, and let $\mu$ be a positive Borel measure on $\D$. Then the following conditions are equivalent:
\begin{enumerate}
    \item [(a)] $I_d:AT^{\infty}_p(\omega)\to T^{\infty}_q(\mu)$ is compact;
    \item [(b)] for any bounded sequence $\{f_j\}\subset AT^{\infty}_p(\omega)$ that converges to $0$ uniformly on compact subsets of $\D$, we have
    $$\lim_{j\to\infty}\|f_j\|_{T^{\infty}_q(\mu)}=0.$$
\end{enumerate}
\end{lemma}

We also need the following lemma.

\begin{lemma}\label{cpt-necessity}
Let $0<p,q<\infty$, $\omega\in\DDD$, and let $\mu$ be a positive Borel measure on $\D$. If either $I_d:AT^{\infty}_p(\omega)\to T^{\infty}_q(\mu)$ or $I_d:AT^0_p(\omega)\to T^0_q(\mu)$ is compact, then $AT^{\infty}_p(\omega)\subset T^0_q(\mu)$.
\end{lemma}
\begin{proof}
Fix $f\in AT^{\infty}_p(\omega)$. Suppose first that $I_d:AT^{\infty}_p(\omega)\to T^{\infty}_q(\mu)$ is compact. Then Corollary \ref{dilation} indicates that $\{f_{\rho}\}_{0<\rho<1}$ is a bounded family in $AT^{\infty}_p(\omega)$, and $f_{\rho}$ converges to $f$ uniformly on compact subsets of $\D$ as $\rho\to1^-$. Hence by Lemma \ref{cpt-ele}, $f_{\rho}\to f$ in $T^{\infty}_q(\mu)$ as $\rho\to1^-$. On the other hand, it follows from Theorems \ref{bdd-em-1} and \ref{bdd-em-2} that for any $\rho\in(0,1)$, $f_{\rho}\in H^{\infty}\subset AT^0_p(\omega)\subset T^0_q(\mu)$. Therefore, $f\in T^0_q(\mu)$ and the embedding $AT^{\infty}_p(\omega)\subset T^0_q(\mu)$ holds.

If $I_d:AT^0_p(\omega)\to T^0_q(\mu)$ is compact, then the above argument gives that there exists a sequence $\rho_j\in(0,1)$ converging to $1$ and a function $g\in T^0_q(\mu)$ such that $f_{\rho_j}\to g$ in $T^0_q(\mu)$ as $j\to\infty$. Bearing in mind that $T^{\infty}_q(\mu)\subset L^q(\mu)$, we can extract a subsequence of $\{f_{\rho_j}\}$ that converges to $g$ $\mu$-almost everywhere on $\D$. Since $f_{\rho}\to f$ pointwisely on $\D$ as $\rho\to1^-$, we have $f=g$ $\mu$-almost everywhere on $\D$. Therefore, $f\in T^0_q(\mu)$ and the embedding $AT^{\infty}_p(\omega)\subset T^0_q(\mu)$ holds.
\end{proof}

We are now ready to characterize the compactness of $I_d:AT^{\infty}_p(\omega)\to T^{\infty}_q(\mu)$ and $I_d:AT^0_p(\omega)\to T^0_q(\mu)$. Let $\mathrm{C}_0(\D)$ be the space of continuous functions $f$ on $\D$ with $f(z)\to0$ as $|z|\to1^-$.

\begin{theorem}\label{em-cpt1}
Let $0<p\leq q<\infty$, $\omega\in\DDD$, and let $\mu$ be a positive Borel measure on $\D$. Then the following conditions are equivalent:
\begin{enumerate}
    \item [(a)] $I_d:AT^{\infty}_p(\omega)\to T^{\infty}_q(\mu)$ is compact;
    \item [(b)] $I_d:AT^0_p(\omega)\to T^0_q(\mu)$ is compact;
    \item [(c)] for some (or any) $r\in(0,1/4)$, $G_{\mu,r}\in\mathrm{C}_0(\D)$.
\end{enumerate}
\end{theorem}
\begin{proof}
Assume that (a) or (b) holds. Then by Theorem \ref{bdd-em-1}, $\mu$ is finite on each compact subset of $\D$, and consequently, $G_{\mu,r}$ is continuous on $\D$ for any $r\in(0,1/4)$. For $z\in\D$, define 
$$f_z(u)=\frac{1}{\widehat{\omega}(z)^{1/p}}\left(\frac{1-|z|}{1-\overline{z}u}\right)^{\gamma},\quad u\in\D,$$
where $\gamma>\big(1+\beta_0(\omega)+\gamma_0(\omega)\big)/p$. Then by Lemma \ref{test-function}, $\{f_z\}$ is a bounded set in $AT^0_p(\omega)$, and by Lemma \ref{elementary}(b), $f_z\to0$ uniformly on compact susbets of $\D$ as $|z|\to1^-$. Hence Lemma \ref{cpt-ele} indicates that $\|f_z\|_{T^{\infty}_q(\mu)}\to0$ as $|z|\to1^-$, which, together with \eqref{bdd-nece}, implies that for any $r\in(0,1/4)$, $G_{\mu,r}\in\mathrm{C}_0(\D)$.

Assume now (c) holds, i.e. $G_{\mu,r}\in\mathrm{C}_0(\D)$ for some $r\in(0,1/4)$. Fix $\epsilon>0$ and an $r$-lattice $\{a_k\}$. Then there exists $K\geq1$ such that
$$G_{\mu,r}(a_k)<\epsilon,\quad \forall k>K.$$
Let $\{f_j\}\subset AT^{\infty}_p(\omega)$ be a bounded sequence converging to $0$ uniformly on compact subsets of $\D$. Then there exists $J\geq1$ such that
$$|f_j(z)|^p<\epsilon^q,\quad \forall z\in\bigcup_{k\leq K}\Delta(a_k,2r)$$
whenever $j>J$. Consequently, for $s>0$ and $j>J$, we deduce from \eqref{suff-embedding} and Lemma \ref{hat-equivalent} that
\begin{align*}
\|f_j\|^q_{T^{\infty}_q(\mu)}
&\lesssim\sup_{a\in\D}\sum_{k\geq1}\frac{(1-|a|)^s}{|1-\overline{a}a_k|^{s+1}}G_{\mu,r}(a_k)^q
  \frac{\widehat{\omega}(a_k)}{1-|a_k|}\int_{\Delta(a_k,2r)}|f_j(u)|^pdA(u)\\
&\leq\epsilon^q\sup_{a\in\D}\sum_{k\leq K}\frac{(1-|a|)^s}{|1-\overline{a}a_k|^{s+1}}G_{\mu,r}(a_k)^q
  \frac{\widehat{\omega}(a_k)}{1-|a_k|}\int_{\Delta(a_k,2r)}dA(u)\\
&\quad+\epsilon^q\sup_{a\in\D}\sum_{k>K}\frac{(1-|a|)^s}{|1-\overline{a}a_k|^{s+1}}
  \frac{\widehat{\omega}(a_k)}{1-|a_k|}\int_{\Delta(a_k,2r)}|f_j(u)|^pdA(u)\\
&\lesssim\epsilon^q\sup_{a\in\D}\int_{\D}\frac{(1-|a|)^s}{|1-\overline{a}u|^{s+1}}
  \frac{\widehat{\omega}(u)}{1-|u|}dA(u)\\
&\quad+\epsilon^q\sup_{a\in\D}\int_{\D}\frac{(1-|a|)^s}{|1-\overline{a}u|^{s+1}}|f_j(u)|^p
  \frac{\widehat{\omega}(u)}{1-|u|}dA(u)\\
&\lesssim\epsilon^q.
\end{align*}
Therefore, $f_j\to0$ in $T^{\infty}_q(\mu)$, which combined with Lemma \ref{cpt-ele} implies (a).

If (c) holds, then $I_d:AT^0_p(\omega)\to T^0_q(\mu)$ is bounded by Theorem \ref{bdd-em-1} and $I_d:AT^{\infty}_p(\omega)\to T^{\infty}_q(\mu)$ is compact by the previous paragraph. Thus, $I_d:AT^0_p(\omega)\to T^0_q(\mu)$ is compact and (b) holds. The proof is complete.
\end{proof}

\begin{theorem}\label{em-cpt2}
Let $0<q<p<\infty$, $\omega\in\DDD$, and let $\mu$ be a positive Borel measure on $\D$. Then the following conditions are equivalent:
\begin{enumerate}
    \item [(a)] $I_d:AT^{\infty}_p(\omega)\to T^{\infty}_q(\mu)$ is compact;
    \item [(b)] $I_d:AT^0_p(\omega)\to T^0_q(\mu)$ is compact;
    \item [(c)] $I_d:AT^{\infty}_p(\omega)\to T^0_q(\mu)$ is bounded;
    \item [(d)] for some (or any) $r\in(0,1/4)$, $G_{\mu,r}\in T^0_{\frac{pq}{p-q},-1}$.
\end{enumerate}
\end{theorem}
\begin{proof}
The implications (a)$\Rightarrow$(c) and (b)$\Rightarrow$(c) follow from Lemma \ref{cpt-necessity} and Theorem \ref{bdd-em-2}. Suppose now that (c) holds. Fix $r\in(0,1/4)$ and $s>0$. Let $Z=\{a_k\}\subset\D$ be an $r$-lattice and define $f_t$ as in \eqref{ft} for $c=\{c_k\}\in T^{\infty}_p(Z)$. Then by Theorem \ref{test-function2} and the condition (c), we know that for almost every $t\in[0,1]$,
$$\lim_{|a|\to1^-}\int_{\D}\frac{(1-|a|)^s}{|1-\overline{a}z|^{s+1}}|f_t(z)|^qd\mu(z)=0.$$
Proceeding as in the proof of Theorem \ref{bdd-em-2}, we can establish that $\{c_kG_{\mu,2r}(a_k)\}\in T^0_q(Z)$. Since $c\in T^{\infty}_p(Z)$ is arbitrary, \cite[Lemma 2.2]{CW22} yields that $\{G_{\mu,2r}(a_k)\}\in T^{0}_{\frac{pq}{p-q}}(Z)$, which together with \eqref{G-inequality} implies that $G_{\mu,r}\in T^0_{\frac{pq}{p-q},-1}$. That is, (d) holds.

Suppose next that (d) holds, i.e. $G_{\mu,r}\in T^0_{\frac{pq}{p-q},-1}$ for some $r\in(0,1/4)$,  and fix $\epsilon>0$. Then by \cite[Lemma 3.1]{CPW}, there exists $\rho\in(0,1)$ such that
$$\sup_{a\in\D}\int_{\D\setminus\rho\D}\frac{(1-|a|)^s}{|1-\overline{a}u|^{s+1}}G_{\mu,r}(u)^{\frac{pq}{p-q}}\frac{dA(u)}{1-|u|}<\epsilon^{\frac{pq}{p-q}}.$$
Let $\{f_j\}\subset AT^{\infty}_p(\omega)$ be a bounded sequence that converges to $0$ uniformly on compact subsets of $\D$. To establish (a), it suffices to show that $\|f_j\|_{T^{\infty}_q(\mu)}\to0$. Since $f_j\to0$ uniformly on $\rho\D$, we may find $J\geq1$ such that $|f_j|<\epsilon$ on $\rho\D$ whenever $j\geq J$. Consequently, for $j\geq J$, similar to \eqref{suff-embedding2}, we have
\begin{align*}
\|f_j\|^q_{T^{\infty}_q(\mu)}
&\lesssim\sup_{a\in\D}\int_{\D}\frac{(1-|a|)^s}{|1-\overline{a}u|^{s+1}}
  |f_j(u)|^q\frac{\widehat{\omega}(u)^{\frac{q}{p}}}{(1-|u|)^{\frac{q}{p}}}G_{\mu,r}(u)^q(1-|u|)^{\frac{q}{p}-1}dA(u)\\
&\leq\sup_{a\in\D}\left(\int_{\rho\D}\frac{(1-|a|)^s}{|1-\overline{a}u|^{s+1}}|f_j(u)|^p
  \frac{\widehat{\omega}(u)}{1-|u|}dA(u)\right)^{\frac{q}{p}}
  \left\|G_{\mu,r}\right\|^q_{T^{\infty}_{\frac{pq}{p-q},-1}}\\
&\quad +\|f_j\|^q_{AT^{\infty}_p(\omega)}\sup_{a\in\D}\left(\int_{\D\setminus\rho\D}
  \frac{(1-|a|)^s}{|1-\overline{a}u|^{s+1}}G_{\mu,r}(u)^{\frac{pq}{p-q}}\frac{dA(u)}{1-|u|}\right)^{\frac{p-q}{p}}\\
&\lesssim\epsilon^q.
\end{align*}
Therefore, (a) holds. In this case, the compactness of $I_d:AT^0_p(\omega)\to T^0_q(\mu)$ follows from (a) and Theorem \ref{bdd-em-2}. The proof is complete.
\end{proof}

\section{Littlewood--Paley formulas and integration operators}\label{LP-Volterra}

In this section, we are going to prove Theorems \ref{Littlewood--Paley} and \ref{bounded-Volterra}.

\begin{proof}[Proof of Theorem \ref{Littlewood--Paley}]
Suppose first that (b) holds, i.e. $\omega\in\DDD$. We are going to establish (a). Fix $a\in\D$ and let $\gamma>\big(m+1+\gamma_0(\omega)\big)\max\{1,1/q\}$ be large enough. Assume that $f\in AT^{\infty}_q(\omega)$. Then the reproducing formula
$$f(z)=(\gamma+1)\int_{\D}\frac{(1-|u|^2)^{\gamma}}{(1-\overline{u}z)^{2+\gamma}}f(u)dA(u)$$
gives that
\begin{equation}\label{m-deri}
|f^{(m)}(z)|\lesssim\int_{\D}\frac{(1-|u|)^\gamma}{|1-\overline{u}z|^{2+m+\gamma}}|f(u)|dA(u).
\end{equation}
In the case $q>1$, we can use H\"older's inequality and Lemma \ref{FR} to obtain that
\begin{align*}
|f^{(m)}(z)|^q
&\lesssim\left(\int_{\D}\frac{(1-|u|)^{\gamma}dA(u)}{|1-\overline{u}z|^{2+m+\gamma}}\right)^{\frac{q}{q'}}
  \int_{\D}\frac{(1-|u|)^{\gamma}}{|1-\overline{u}z|^{2+m+\gamma}}|f(u)|^qdA(u)\\
&\lesssim(1-|z|)^{-m\frac{q}{q'}}\int_{\D}\frac{(1-|u|)^{\gamma}}{|1-\overline{u}z|^{2+m+\gamma}}|f(u)|^qdA(u).
\end{align*}
Consequently, for sufficiently small $s>0$, Fubini's theorem together with Lemmas \ref{gFR}(2) and \ref{hat-equivalent} yields that
\begin{align*}
&\int_{\D}\frac{(1-|a|)^s}{|1-\overline{a}z|^{s+1}}|f^{(m)}(z)|^q(1-|z|)^{mq}\omega(z)dA(z)\\
&\ \lesssim\int_{\D}\frac{(1-|a|)^s}{|1-\overline{a}z|^{s+1}}\int_{\D}
  \frac{(1-|u|)^{\gamma}}{|1-\overline{u}z|^{2+m+\gamma}}|f(u)|^qdA(u)(1-|z|)^m\omega(z)dA(z)\\
&\ =\int_{\D}(1-|a|)^s(1-|u|)^{\gamma}|f(u)|^q\int_{\D}\frac{(1-|z|)^m\omega(z)dA(z)}
  {|1-\overline{a}z|^{s+1}|1-\overline{u}z|^{2+m+\gamma}}dA(u)\\
&\ \lesssim\int_{\D}\frac{(1-|a|)^s}{|1-\overline{a}u|^{s+1}}|f(u)|^q\frac{\widehat{\omega}(u)}{1-|u|}dA(u)\\
&\ \asymp\int_{\D}\frac{(1-|a|)^s}{|1-\overline{a}u|^{s+1}}|f(u)|^q\omega(u)dA(u).
\end{align*}
In the case $0<q\leq1$, the estimate \eqref{m-deri} together with \cite[Proposition 4.17]{Zh07} gives that
$$|f^{(m)}(z)|^q\lesssim\int_{\D}\frac{(1-|u|)^{q(2+\gamma)-2}}{|1-\overline{u}z|^{q(2+m+\gamma)}}|f(u)|^qdA(u).$$
Then, for sufficiently small $s>0$, by Fubini's theorem, Lemmas \ref{gFR}(2) and \ref{hat-equivalent},
\begin{align*}
&\int_{\D}\frac{(1-|a|)^s}{|1-\overline{a}z|^{s+1}}|f^{(m)}(z)|^q(1-|z|)^{mq}\omega(z)dA(z)\\
&\lesssim\int_{\D}(1-|a|)^s(1-|u|)^{q(2+\gamma)-2}|f(u)|^q\int_{\D}\frac{(1-|z|)^{mq}\omega(z)dA(z)}
  {|1-\overline{a}z|^{s+1}|1-\overline{u}z|^{q(2+m+\gamma)}}dA(u)\\
&\lesssim\int_{\D}\frac{(1-|a|)^s}{|1-\overline{a}u|^{s+1}}|f(u)|^q\frac{\widehat{\omega}(u)}{1-|u|}dA(u)\\
&\asymp\int_{\D}\frac{(1-|a|)^s}{|1-\overline{a}u|^{s+1}}|f(u)|^q\omega(u)dA(u).
\end{align*}
Since $a\in\D$ is arbitrary, we conclude that in both cases, the function $f^{(m)}(\cdot)(1-|\cdot|)^m$ belongs to $T^{\infty}_q(\omega)$. Moreover,
$$\left\|f^{(m)}(\cdot)(1-|\cdot|)^m\right\|_{T^{\infty}_q(\omega)}\lesssim\|f\|_{AT^{\infty}_q(\omega)}.$$
Combining this with Lemma \ref{thm:growth-estimate-unit-disk}, we obtain that
$$\sum_{j=0}^{m-1}\left|f^{(j)}(0)\right|+\left\|f^{(m)}(\cdot)(1-|\cdot|)^m\right\|_{T^{\infty}_q(\omega)}
\lesssim\|f\|_{AT^{\infty}_q(\omega)}.$$

Assume now that $f$ is an analytic function on $\D$ such that $f^{(m)}(\cdot)(1-|\cdot|)^m$ belongs to $T^{\infty}_q(\omega)$. Define
$$g(z)=f(z)-\sum_{j=0}^{2m}\frac{f^{(j)}(0)}{j!}z^j.$$
It is clear that any bounded function belongs to $T^{\infty}_q(\omega)$, and by Lemma \ref{thm:growth-estimate-unit-disk},
\begin{equation}\label{j-deritive-0}
|f^{(j)}(0)|\lesssim\left\|f^{(m)}(\cdot)(1-|\cdot|)^m\right\|_{T^{\infty}_q(\omega)},\quad m\leq j\leq2m.
\end{equation}
Then we have $g^{(m)}(\cdot)(1-|\cdot|)^m$ is in $T^{\infty}_q(\omega)$, and
\begin{align*}
\left\|g^{(m)}(\cdot)(1-|\cdot|)^m\right\|_{T^{\infty}_q(\omega)}
&\lesssim\left\|f^{(m)}(\cdot)(1-|\cdot|)^m\right\|_{T^{\infty}_q(\omega)}+
  \sum_{j=m}^{2m}\left|f^{(j)}(0)\right|\\
&\lesssim\left\|f^{(m)}(\cdot)(1-|\cdot|)^m\right\|_{T^{\infty}_q(\omega)}.
\end{align*}
Since
$$g(0)=g'(0)=\cdots=g^{(2m)}(0)=0,$$
we may apply \cite[Proposition 4.27]{Zh07} to obtain that
$$g(z)=\frac{1}{(\gamma+2)\cdots(\gamma+m)}\int_{\D}\frac{(1-|u|^2)^{m+\gamma}}{\overline{u}^m(1-\overline{u}z)^{2+\gamma}}g^{(m)}(u)dA(u),\quad z\in\D.$$
Noting that the function $u\mapsto u^{-m}g^{(m)}(u)$ is analytic on $\D$, the above representation together with \cite[Lemma 4.26]{Zh07} gives that
\begin{equation}\label{m-inte}
|g(z)|\lesssim\int_{\D}\frac{(1-|u|)^{m+\gamma}}{|1-\overline{u}z|^{2+\gamma}}\left|g^{(m)}(u)\right|dA(u).
\end{equation}
In the case $q>1$, by H\"older's inequality and Lemma \ref{FR}, we have for sufficiently small $\epsilon>0$,
\begin{align*}
|g(z)|^q
&\lesssim\left(\int_{\D}\frac{(1-|u|)^{\gamma-\epsilon q'}dA(u)}{|1-\overline{u}z|^{2+\gamma}}\right)^{\frac{q}{q'}}
  \int_{\D}\frac{(1-|u|)^{mq+\gamma+\epsilon q}}{|1-\overline{u}z|^{2+\gamma}}\left|g^{(m)}(u)\right|^qdA(u)\\
&\lesssim(1-|z|)^{-\epsilon q}
  \int_{\D}\frac{(1-|u|)^{mq+\gamma+\epsilon q}}{|1-\overline{u}z|^{2+\gamma}}\left|g^{(m)}(u)\right|^qdA(u).
\end{align*}
Then for sufficinetly small $s>0$, by Fubini's theorem, Lemmas \ref{gFR}(2) and \ref{hat-equivalent},
\begin{align*}
&\int_{\D}\frac{(1-|a|)^s}{|1-\overline{a}z|^{s+1}}|g(z)|^q\omega(z)dA(z)\\
&\lesssim\int_{\D}(1-|a|)^s(1-|u|)^{mq+\gamma+\epsilon q}\left|g^{(m)}(u)\right|^q
  \int_{\D}\frac{(1-|z|)^{-\epsilon q}\omega(z)dA(z)}{|1-\overline{a}z|^{s+1}|1-\overline{u}z|^{2+\gamma}}dA(u)\\
&\lesssim\int_{\D}\frac{(1-|a|)^s}{|1-\overline{a}u|^{s+1}}\left|g^{(m)}(u)\right|^q(1-|u|)^{mq-1}
  \widehat{\omega}(u)dA(u)\\
&\asymp\int_{\D}\frac{(1-|a|)^s}{|1-\overline{a}u|^{s+1}}\left|g^{(m)}(u)\right|^q(1-|u|)^{mq}
  \omega(u)dA(u)\\
&\lesssim\left\|g^{(m)}(\cdot)(1-|\cdot|)^m\right\|^q_{T^{\infty}_q(\omega)}\lesssim\left\|f^{(m)}(\cdot)(1-|\cdot|)^m\right\|^q_{T^{\infty}_q(\omega)}.
\end{align*}
In the case $0<q<1$, by \eqref{m-inte} and \cite[Proposition 4.17]{Zh07},
$$|g(z)|^q\lesssim\int_{\D}\frac{(1-|u|)^{(2+\gamma+m)q-2}}{|1-\overline{u}z|^{(2+\gamma)q}}
\left|g^{(m)}(u)\right|^qdA(u),$$
and similarly as before, for sufficiently small $s>0$,
\begin{align*}
&\int_{\D}\frac{(1-|a|)^s}{|1-\overline{a}z|^{s+1}}|g(z)|^q\omega(z)dA(z)\\
&\lesssim\int_{\D}(1-|a|)^s(1-|u|)^{(2+\gamma+m)q-2}\left|g^{(m)}(u)\right|^q
  \int_{\D}\frac{\omega(z)dA(z)}{|1-\overline{a}z|^{s+1}|1-\overline{u}z|^{(2+\gamma)q}}dA(u)\\
&\lesssim\int_{\D}\frac{(1-|a|)^s}{|1-\overline{a}u|^{s+1}}\left|g^{(m)}(u)\right|^q(1-|u|)^{mq-1}
  \widehat{\omega}(u)dA(u)\\
&\asymp\int_{\D}\frac{(1-|a|)^s}{|1-\overline{a}u|^{s+1}}\left|g^{(m)}(u)\right|^q(1-|u|)^{mq}
  \omega(u)dA(u)\\
&\lesssim\left\|g^{(m)}(\cdot)(1-|\cdot|)^m\right\|^q_{T^{\infty}_q(\omega)}\lesssim\left\|f^{(m)}(\cdot)(1-|\cdot|)^m\right\|^q_{T^{\infty}_q(\omega)}.
\end{align*}
Consequently, in both cases, $g\in AT^{\infty}_q(\omega)$, and
$$\|g\|_{AT^{\infty}_{q}(\omega)}\lesssim\left\|f^{(m)}(\cdot)(1-|\cdot|)^m\right\|_{T^{\infty}_q(\omega)}.$$
Therefore, by the definition of $g$ and \eqref{j-deritive-0}, we obtain that $f\in AT^{\infty}_q(\omega)$, and
\begin{align*}
\|f\|_{AT^{\infty}_q(\omega)}
&=\left\|g(\cdot)+\sum_{j=0}^{2m}\frac{f^{(j)}(0)}{j!}(\cdot)^j\right\|_{AT^{\infty}_q(\omega)}\lesssim\|g\|_{AT^{\infty}_q(\omega)}+\sum_{j=0}^{2m}\left|f^{(j)}(0)\right|\\
&\lesssim\sum_{j=0}^{m-1}\left|f^{(j)}(0)\right|+\left\|f^{(m)}(\cdot)(1-|\cdot|)^m\right\|_{T^{\infty}_q(\omega)}.
\end{align*}

Suppose next that (a) holds. We are going to show $\omega \in \DDD$. For $x>0$, write
$$\omega_x:=\int_0^1r^x\omega(r)dr.$$
Consider the test functions $f_n(z)=z^n$ for $n\in\mathbb{N}$. An elementary computation gives that
$$\|f_n\|_{AT_q^\infty(\omega)}^q \asymp \omega_{nq+1},$$
and
$$\left\|f'_n(\cdot)(1-|\cdot|)\right\|^q_{T^{\infty}_q(\omega)}\asymp
n^q\int_0^1r^{(n-1)q+1}(1-r)^q\omega(r)dr,$$
where the implicit constants are both independent of $n$. Then by the condition (a), for any $n\in\mathbb{N}$,
$$n^q\int_0^1r^{(n-1)q+1}(1-r)^q\omega(r)dr\asymp\omega_{nq+1}.$$
It is clear that $\omega_{nq+1}\leq\omega_{(n-1)q+1}$, we obtain that there exists $C>0$ such that
$$n^q\int_0^1r^{(n-1)q+1}(1-r)^q\omega(r)dr\leq C\omega_{(n-1)q+1},\quad \forall n\in\mathbb{N},$$
which gives $\omega\in\DD$ (see the proof of \cite[Theorem 6]{PR21}). Consequently,
$$\omega_{(n-1)q+1}\asymp\omega_{nq+1}\asymp n^q\int_0^1r^{(n-1)q+1}(1-r)^q\omega(r)dr$$
holds for any $n\in\mathbb{N}$ with implicit constants independent of $n$, which in turn implies that $\omega\in\DDD$ (see the proof of \cite[Theorem 5]{PR21}). The general case $m\in \N$ readily follows from the argument above. The proof is finished.
\end{proof}

The following derivative characterization of the small spaces $AT^0_q(\omega)$ is immediate by the previous proof.

\begin{corollary}\label{LP0}
Let $0<q<\infty$, $\omega\in\DDD$, $m\in\mathbb{N}$, and let $f$ be an analytic function on $\D$. Then $f\in AT^0_q(\omega)$ if and only if the function $f^{(m)}(\cdot)(1-|\cdot|)^m$ belongs to $T^0_q(\omega)$.
\end{corollary}

We are now in a position to prove Theorem \ref{bounded-Volterra}.

\begin{proof}[Proof of Theorem \ref{bounded-Volterra}]
Fix $s>0$. For any $f\in AT^{\infty}_p(\omega)$, we deduce from Theorem \ref{Littlewood--Paley} and Lemma \ref{hat-equivalent} that
\begin{align}\label{Jgf-norm}
\|J_gf\|^q_{AT^{\infty}_q(\omega)}
&\asymp\sup_{a\in\D}\int_{\D}\frac{(1-|a|)^s}{|1-\overline{a}z|^{s+1}}|f(z)g'(z)|^q(1-|z|)^q\omega(z)dA(z)\\
&\asymp\sup_{a\in\D}\int_{\D}\frac{(1-|a|)^s}{|1-\overline{a}z|^{s+1}}|f(z)g'(z)|^q
    (1-|z|)^{q-1}\widehat{\omega}(z)dA(z)\nonumber\\
&\asymp\|f\|^q_{T^{\infty}_q(\mu_g)},\nonumber
\end{align}
where
$$d\mu_g(z):=|g'(z)|^q(1-|z|)^{q-1}\widehat{\omega}(z)dA(z).$$
Consequently, $J_g:AT^{\infty}_p(\omega)\to AT^{\infty}_q(\omega)$ is bounded if and only if the embedding $I_d:AT^{\infty}_p(\omega)\to T^{\infty}_q(\mu_g)$ is bounded. Moreover,
\begin{equation}\label{JgId-equivalent}
\|J_g\|_{AT^{\infty}_p(\omega)\to AT^{\infty}_q(\omega)}
\asymp\|I_d\|_{AT^{\infty}_p(\omega)\to T^{\infty}_q(\mu_g)}.
\end{equation}
Similarly, by Corollary \ref{LP0}, $J_g:AT^{0}_p(\omega)\to AT^{0}_q(\omega)$ is bounded if and only if the embedding $I_d:AT^{0}_p(\omega)\to T^{0}_q(\mu_g)$ is bounded. Therefore, by Theorem \ref{bounded-embedding}, in both cases the conditions (a) and (b) are equivalent. Note that for any $r\in(0,1/4)$, \eqref{D-constant} together with the sub-harmonic property gives that
\begin{equation}\label{lower-bd-G}
G_{\mu_g,r}(z)=\frac{\mu_g(\Delta(z,r))^{\frac{1}{q}}}{\widehat{\omega}(z)^{\frac{1}{p}}(1-|z|)^{\frac{1}{q}}}
\gtrsim|g'(z)|(1-|z|)\widehat{\omega}(z)^{\frac{1}{q}-\frac{1}{p}}.
\end{equation}

(1) If (a) holds, then by \eqref{JgId-equivalent}, \eqref{lower-bd-G} and Theorem \ref{bounded-embedding}, we obtain that (c) holds, and
$$\|J_g\|_{AT^{\infty}_p(\omega)\to AT^{\infty}_q(\omega)}\gtrsim\sup_{z\in\D}|g'(z)|(1-|z|)\widehat{\omega}(z)^{\frac{1}{q}-\frac{1}{p}}.$$
Conversely, if (c) holds, then for any $z\in\D$, by \eqref{D-constant},
\begin{align*}
G_{\mu_g,r}(z)&=\frac{1}{\widehat{\omega}(z)^{\frac{1}{p}}(1-|z|)^{\frac{1}{q}}}
  \left(\int_{\Delta(z,r)}|g'(u)|^q(1-|u|)^{q-1}\widehat{\omega}(u)dA(u)\right)^{\frac{1}{q}}\\
&\leq\frac{\left(\int_{\Delta(z,r)}(1-|u|)^{-1}\widehat{\omega}(u)^{\frac{q}{p}}dA(u)\right)^{\frac{1}{q}}}  
  {\widehat{\omega}(z)^{\frac{1}{p}}(1-|z|)^{\frac{1}{q}}}
  \cdot\sup_{z\in\D}|g'(z)|(1-|z|)\widehat{\omega}(z)^{\frac{1}{q}-\frac{1}{p}}\\
&\lesssim\sup_{z\in\D}|g'(z)|(1-|z|)\widehat{\omega}(z)^{\frac{1}{q}-\frac{1}{p}}.
\end{align*}
Hence by Theorem \ref{bounded-embedding} and \eqref{JgId-equivalent}, we obtain that (a) holds, and
$$\|J_g\|_{AT^{\infty}_p(\omega)\to AT^{\infty}_q(\omega)}\lesssim\sup_{z\in\D}|g'(z)|(1-|z|)\widehat{\omega}(z)^{\frac{1}{q}-\frac{1}{p}}.$$

(2) If (a) holds, then by \eqref{JgId-equivalent}, \eqref{lower-bd-G} and Theorem \ref{bounded-embedding}, we obtain that the function $z\mapsto|g'(z)|(1-|z|)\widehat{\omega}(z)^{\frac{1}{q}-\frac{1}{p}}$ belongs to $T^{\infty}_{\frac{pq}{p-q},-1}$, which, due to Lemma \ref{hat-equivalent} and Theorem \ref{Littlewood--Paley}, is equivalent to $g\in AT^{\infty}_{\frac{pq}{p-q}}(\omega)$. Moreover,
\begin{align*}
\|J_g\|_{AT^{\infty}_p(\omega)\to AT^{\infty}_q(\omega)}
&\gtrsim\left\||g'(\cdot)|(1-|\cdot|)
    \widehat{\omega}(\cdot)^{\frac{1}{q}-\frac{1}{p}}\right\|_{T^{\infty}_{\frac{pq}{p-q},-1}}\\
&\asymp\|g-g(0)\|_{AT^{\infty}_{\frac{pq}{p-q}}(\omega)}.
\end{align*}
Conversely, if (c) holds, then for any $f\in AT^{\infty}_p(\omega)$, by \eqref{Jgf-norm}, H\"{o}lder's inequality and Theorem \ref{Littlewood--Paley},
\begin{align*}
\|J_gf\|^q_{AT^{\infty}_q(\omega)}
&\asymp\sup_{a\in\D}\int_{\D}\frac{(1-|a|)^s}{|1-\overline{a}z|^{s+1}}|f(z)g'(z)|^q(1-|z|)^q\omega(z)dA(z)\\
&\leq\left(\sup_{a\in\D}\int_{\D}\frac{(1-|a|)^s}{|1-\overline{a}z|^{s+1}}|f(z)|^p\omega(z)
    dA(z)\right)^{\frac{q}{p}}\\
&\  \cdot\left(\sup_{a\in\D}\int_{\D}\frac{(1-|a|)^s}{|1-\overline{a}z|^{s+1}}
    |g'(z)|^{\frac{pq}{p-q}}(1-|z|)^{\frac{pq}{p-q}}\omega(z)dA(z)\right)^{\frac{p-q}{p}}\\
&\asymp\|f\|^q_{AT^{\infty}_{p}(\omega)}\cdot\|g-g(0)\|^q_{AT^{\infty}_{\frac{pq}{p-q}}(\omega)}.
\end{align*}
Therefore, $J_g:AT^{\infty}_p(\omega)\to AT^{\infty}_q(\omega)$ is bounded, and
$$\|J_g\|_{AT^{\infty}_p(\omega)\to AT^{\infty}_q(\omega)}
\lesssim\|g-g(0)\|_{AT^{\infty}_{\frac{pq}{p-q}}(\omega)}.$$
The proof is complete.
\end{proof}

Using the same method and Theorems \ref{em-cpt1}, \ref{em-cpt2}, we can also characterize the compactness of Volterra type integration operators $J_g$. The details are omitted.

\begin{theorem}
Let $0<p,q<\infty$, $\omega\in\DDD$, and let $g$ be an analytic function on $\D$.
\begin{enumerate}
    \item [(1)] If $p\leq q$, then the following conditions are equivalent:
      \begin{enumerate}
        \item [(a)] $J_g:AT^{\infty}_p(\omega)\to AT^{\infty}_q(\omega)$ is compact;
        \item [(b)] $J_g:AT^0_p(\omega)\to AT^0_q(\omega)$ is compact;
        \item [(c)] $\lim_{|z|\to1^-}\left|g'(z)\right|(1-|z|)\widehat{\omega}(z)^{\frac{1}{q}-\frac{1}{p}}=0$.
      \end{enumerate}
    \item [(2)] If $p>q$, then the following conditions are equivalent:
      \begin{enumerate}
          \item [(a)] $J_g:AT^{\infty}_p(\omega)\to AT^{\infty}_q(\omega)$ is compact;
          \item [(b)] $J_g:AT^0_p(\omega)\to AT^0_q(\omega)$ is compact;
          \item [(c)] $J_g:AT^{\infty}_p(\omega)\to AT^0_q(\omega)$ is bounded;
          \item [(d)] $g\in AT^0_{\frac{pq}{p-q}}(\omega)$.
      \end{enumerate}
\end{enumerate}
\end{theorem}

\end{document}